\documentclass[11pt, oneside]{article}   	
\usepackage{geometry}    
\usepackage{bbold}
\usepackage[dvipsnames]{xcolor}
\usepackage{authblk}
\usepackage{cancel} 
\usepackage{ulem} 

\usepackage{subfigure}
\usepackage{amssymb,amsmath,amsthm}
\usepackage[colorlinks=true, citecolor=blue, urlcolor=blue, linkcolor=red]{hyperref}
\usepackage{float}
\usepackage{fullpage}
\usepackage{empheq}

\newcommand{\Dc}{\mathcal{D}}
\newcommand{\Ec}{\mathcal{E}}
\newcommand{\Nbb}{\mathbb{N}}
\newcommand{\Pbb}{\mathbb{P}}
\newcommand{\Rbb}{\mathbb{R}}

\newtheorem{theorem}{Theorem}[section]
\newtheorem{remark}[theorem]{Remark}
\newtheorem{lemma}[theorem]{Lemma}
\newtheorem{proposition}[theorem]{Proposition}
\newtheorem{example}[theorem]{ Example}

\usepackage{caption}
\usepackage{tabularx}

\newcommand{\red}[1]{\textcolor{black}{#1}}
\newcommand{\blue}[1]{\textcolor{black}{#1}}
\newcommand{\magenta}[1]{\textcolor{black}{#1}}

\title{Nonlinear manifold approximation using  compositional polynomial networks}

\author[1]{Antoine Bensalah}
\author[2]{Anthony Nouy}
\author[1,2]{Joel Soffo}

\affil[1]{Airbus}
\affil[2]{Centrale Nantes, Nantes Universit\'e, LMJL, CNRS UMR 6629, France}

\date{}							

\begin{document}
\maketitle

\begin{abstract}
We consider the problem of approximating a  subset $M$ of a Hilbert space $X$ by a low-dimensional manifold $M_n$, using samples from $M$. We propose a nonlinear approximation method 
where $M_n $ is defined as the range of a smooth nonlinear decoder $D$ defined on $\mathbb{R}^n$ with values in a possibly high-dimensional linear space $X_N$,  and a linear encoder $E$ which associates to an element from $ M$ its coefficients $E(u)$ on a basis of a $n$-dimensional subspace $X_n \subset X_N$, where \red{$X_N$ is an optimal or near to optimal linear space}, depending on the selected error measure
The linearity of the encoder allows to easily obtain the parameters $E(u)$ associated with a given element $u$ in $M$. 
The proposed  decoder is a polynomial map from  $\mathbb{R}^n$ to $X_N$ which is obtained by a tree-structured composition of polynomial maps, estimated sequentially from samples in $M$. 
Rigorous error and stability analyses are provided, as well as an adaptive strategy for constructing \red{the subspace $X_n$, and} a decoder that guarantees an approximation of the set $M$ with controlled   mean-squared or wort-case errors, and a controlled stability (Lipschitz continuity) of the encoder and decoder pair. We demonstrate the performance of our method through numerical experiments.

\end{abstract}

\section{Introduction}

We consider the problem of approximating a  subset $M$ of a Hilbert space $X$ by some low-dimensional manifold, using samples from $M$. \magenta{Typical examples include the case where $M$ is the image of a Hilbert-valued random variable or the set of solutions of a parameter-dependent equation.}
A large class of manifold approximation (or dimension reduction) methods can be described by an encoder $E : M  \to \mathbb{R}^n$ and a decoder 
$D : \mathbb{R}^n \to X$. The decoder provides a parametrization of a $n$-dimensional manifold 
$$
M_n  = \{ D(a) : a \in \mathbb{R}^n\},
$$
while the encoder is associated with the approximation process, and associates to an element $u \in M$ a parameter value $a = E(u) \in \mathbb{R}^n$. 
An element $u\in M$ is approximated by $D\circ E(u) \in M_n$. This problem is equivalent to approximating the identity map on $M$ by a composition $D\circ E$, which is sometimes called an auto-encoder of $M$. 

Linear methods consist in approximating $M$ by a $n$-dimensional linear space $M_n$, e.g. constructed by greedy algorithms in reduced basis methods, with a control of a worst-case error over $M$, or proper orthogonal decomposition (POD) or principal component analysis (PCA) for a control of a mean-squared error over $M$, see  \cite{book_rbm,Quarteroni2015ReducedBM,morbook2017}. Although linear methods have been proved to be efficient for numerous cases, there are many situations \red{(such as transport-dominated parameter-dependent equations)} where the required dimension $n$ to obtain a good approximation of $ M $ \red{has to be large}, which yields a prohibitive computational cost when it comes to using $M_n$ for different online tasks such as inverse problems, optimization, uncertainty quantification... To  overcome limitations  of linear methods,  several nonlinear dimension reduction methods have been introduced. First, different approaches consider for $M_n$ a union of linear or affine spaces \cite{farhat_localrobs, certified_hp}, that can be obtained by partitioning  the set $M$ into different subsets and approximating each subset by a linear or affine subspace of fixed or variable dimension. 
In this context, adaptive online strategies were proposed in \cite{Carlberg_2014,peherstorfer2020modelreductiontransportdominatedproblems} with real-time applications. Different neural networks architectures were also proposed, as in \cite{Fresca2021May,Kadeethum_2022,KIM2022110841,Lee2020Mar}, where encoders and decoders are both neural networks. 
Such approaches have been proved to give accurate results in various applications, but suffer from scalability issues, as they involve the dimension of the high dimensional space. Also, learning the nonlinear maps is a difficult task, which prevents to achieve precision that is often requested in practical applications from computational science.    

In this work, following \cite{Barnett_2022,Geelen_2023,Barnett2023Nov,geelen2024}, we propose a nonlinear approximation method using a nonlinear decoder with values in a linear/affine space $X_N$ of dimension $N \ge n$, and a linear encoder which associates to an element $u \in M$ its coefficients on a basis of a $n$-dimensional subspace $X_n \subset X_N$, where \red{$X_N$ is an optimal or near-optimal space}, e.g. obtained by PCA or greedy algorithms depending on the desired control of the error. 
The linearity of the encoder allows to easily obtain the parameters $a = E(u)$ associated with a given element $u$ in $M$. 
The proposed  decoder is a polynomial map from  $\mathbb{R}^n$ to $X_N$ which is obtained by a tree-structured composition of polynomial maps, estimated from samples in $M$. More precisely, \blue{given an element $\bar u\in X$ and basis $\varphi_1, \hdots, \varphi_N$ of $X_N$}, the proposed decoder takes the form 
\blue{
$$
D(a) = \bar u + \sum_{i \in I} a_i \varphi_i + \sum_{i \in I^c} g_i(a) \varphi_i,
$$
where $I $ is a subset of $\{1,\dots , N\}$ of cardinal $n$, $I^c$ is its complementary set, and the maps $g_i$ are defined recursively using compositions of polynomial maps.
 The functions $\{\varphi_i : i\in I\}$ form a basis of $X_n$.} Rigorous error and stability analyses are provided, as well as an adaptive strategy for constructing a decoder that guarantees an approximation of the manifold $M$ with controlled error in mean-squared or wort-case settings, and a controlled stability (Lipschitz continuity) of the encoder and decoder pair. 
\red{Note that in general, choosing for $X_n$ an optimal space for linear approximation (i.e. $I=\{1,\dots,n\}$ when the functions $(\varphi_i)_{i\ge 1}$ form an optimal hierarchy for linear approximation) is usually not optimal when using nonlinear decoders. This has been demonstrated in \cite{schwerdtner2024greedyconstructionquadraticmanifolds} for quadratic manifold approximation, where the authors propose a greedy algorithm for the construction of a space $X_n$ from $n$ suitably selected elements of a basis of $X_N$.  Our approach follows the same line, and selects the index set $I$ in an adaptive manner based on approximation and stability properties.
}
The performance of our method is demonstrated through numerical experiments.

The rest of the paper is organized as follows. In Section \ref{sec:manifold-approx}, we present different manifold approximation methods, from classical linear methods to more recent nonlinear methods, with an encoder-decoder point of view. We discuss limitations of state of the art methods and motivate the introduction of a new approach based on compositions of functions. In Section \ref{sec:method}, we present a new nonlinear method with a linear encoder and a nonlinear decoder based on compositions of functions. We present an error analysis in  worst-case and mean-squared settings, and a stability analysis.
Based on the previous analyses, we propose in   Section \ref{sec:algorithm}  an adaptive algorithm for the construction of the encoder-decoder pair, which guarantees to approximate the manifold with a prescribed error and Lipschitz constant of the auto-encoder. In Section \ref{sec:experiments}, we illustrate the performance of our approach through numerical experiments.

\section{Manifold approximation and related $n$-widths}\label{sec:manifold-approx}

Dimension reduction methods can be classified in terms of the properties of their encoders and decoders. The optimal performance   of a given class  $\Ec_n$ of encoders   from $X$ to $\Rbb^n$  and a given class  $\Dc_n$ of decoders from $\Rbb^n \to X$ can be assessed  in worst-case setting  by 
 $$
w(M ; \Ec_n,\Dc_n)_X = \inf_{D\in \Dc_n , E \in \Ec_n}  \sup_{u \in M} \Vert u- D\circ E(u) \Vert_X.
$$
This defines a notion of width of the set $M$. If the set $M$ is equipped with a Borel measure $\rho$ with finite order $p$ moment, $p>0$, the optimal performance can be measured in $p$-average sense by 
$$
w^{(p)}(M , \rho \, ; \Ec_n,\Dc_n)_X = \inf_{D\in \Dc_n , E \in \Ec_n} \left( \int_M \Vert u- D\circ E(u) \Vert_X^p d\rho(u) \right)^{1/p}.
$$

\subsection{Linear approximation}\label{sec:lin-approx}

When the decoder $D$ is a linear map, its range $M_n$ is a linear space with dimension at most $n$. This corresponds to a  linear approximation of the set $M$.
Restricting both the decoder and the encoder to be linear maps yields the approximation numbers 
$$
a_n(M)_X = \inf_{D\in  L( \Rbb^n ; X) , E \in  L(X ; \Rbb^n)}  \sup_{u \in M} \Vert u- D\circ E(u) \Vert_X = \inf_{\mathrm{rank}(A)=n}  \sup_{u \in M} \Vert u- A u \Vert_X,
$$  
where the infimum is taken over all linear maps $A \in L(X;X)$ with rank $n$. This 
 characterizes the optimal performance of linear  methods. A set $M$ is relatively compact if and only if $a_n(M)_X$ converges to $0$ as $n$ goes to infinity \blue{\cite{pinkus2012n}.}

Restricting only the decoder to be a linear map yields the  Kolmogorov $n$-width 
$$
d_n(M)_X = \inf_{D \in L(\Rbb^n ; X)} \sup_{u \in M} \inf_{a \in \Rbb^n} \Vert u- D(a) \Vert_X = \inf_{\dim{M_n}=n}   \sup_{u \in M} \inf_{v\in M_n} \Vert u-v\Vert_X,
$$
which measures how well $M$ can be approximated by a $n$-dimensional space. When $X$ is a Hilbert space, an optimal decoder-encoder pair $(D,E)$ corresponds  to the orthogonal projection $P_{M_n}$ from  $X$ onto an optimal $n$-dimensional space $M_n$. 
Given a basis $\varphi_1, \hdots,\varphi_n$ of an optimal space  $M_n$, the optimal decoder is $D(a) = \sum_{i=1}^n a_i \varphi_i$ and the associated optimal  encoder  provides the coefficients in $\Rbb^n$ of  the orthogonal  projection $P_{M_n} u$, which is a linear and continuous map. Therefore, $d_n(M)_X = a_n(M)_X$ when $X$ is a Hilbert space. However, for a Banach space $X$,  the optimal encoder is possibly nonlinear and non-continuous, but $d_n(M)_X \le a_n(M)_X \le \sqrt{n} d_n(M)_X $ \cite{pinkus2012n}.

In practice, optimal linear spaces in worst-case setting are out of reach but near-optimal spaces $M_n$ can be obtained by greedy algorithms \cite{DeVore:2013fk}, that generate spaces from samples in $M$. They consist in constructing a sequence of nested spaces $0 = M_0 \subset M_1 \subset \dots \subset M_n \subset \dots$ such that $M_n = \mathrm{span}\{v_1, \hdots, v_n\}$ and   $v_{n+1}$ is an element from $M$ which is not well approximated by $M_n$, more precisely
$$
\inf_{v \in M_n} \Vert v_{n+1} - v \Vert_X \ge \gamma \sup_{u\in M} \inf_{v \in M_n} \Vert u - v \Vert_X,
$$
for some fixed constant $0 < \gamma \le 1.$ The performance of this algorithm has been studied in \cite{DeVore:2013fk}. In practice, the set $M$ can be replaced by a (large) finite set of samples $u_1, \hdots, u_m$ from  $M$,  and the $v_i$ are selected from this finite set, see \cite{cohen2020reduced}  for guaranteed approaches using sequential random sampling in $M$.

When $X$ is a Hilbert space and the error is measured in $p$-average sense with $p=2$ (mean-squared error), it yields a $\rho$-average Kolmogorov $n$-width  $d_n^{(2)}(M,\rho)_X$ 
such that 
$$
d_n^{(2)}(M, \rho)_X^2 := \inf_{D \in L(\Rbb^n; X) } \int_M \inf_{a \in \Rbb^n} \Vert u- D(a) \Vert_X^2d\rho(u)  = \inf_{\dim{M_n}=n}  \int_{M} \Vert u- P_{M_n} u \Vert_X^2 d\rho(u).
$$
\blue{Under some assumptions on the measure $\rho$,} an optimal space $M_n$ is given by a dominant eigenspace of a  compact operator $T : X\to X$ defined by
$$
T : v\mapsto \int_M u (u , v)_X d\rho(u),
$$
and 
$$
d_n^{(2)}(M,\rho)_X = \sqrt{\sum_{i>n} \lambda_i } \, ,
$$
where $\{\lambda_i \}_{i\ge 1}$ is the sequence of eigenvalues of $T$, sorted by decreasing values. This is equivalent to finding the $n$ dominant singular values and right dominant singular space of the operator $U : v \in X \mapsto (u,v)_{X} \in L^2_\rho(X)$, such that $T = U^* U$ with $U^* : g \in L^2_\rho\mapsto \int_M u g(u) d\rho(u) \in X $ the dual of $U$. \blue{A classical setting where the above decomposition holds is when $\rho$ is the distribution of a second order strongly measurable random variable with values in $X$, which is the classical framework for functional principal component analysis, or the push-forward measure of a Bochner square integrable function from a set equipped with a finite measure to $X$.}

An optimal space can be estimated in practice from samples $u_1, \hdots, u_m$ in $M$ by solving 
$$
\inf_{\dim{M_n}=n}  \sum_{i=1}^m \Vert u_i- P_{M_n} u_i \Vert_X^2,
$$
whose solution is the  dominant eigenspace of the operator
$
\hat T : v\mapsto \sum_{i=1}^m u_i (u_i , v)_X.
$
This is equivalent to computing the singular value decomposition of the operator 
$
\hat U : x \in \Rbb^m \mapsto \sum_{i=1}^m x_i u_i \in X,
$  
whose dual is $\hat U^* : v\in X \mapsto ((u_i,v)_X)_{i=1}^m \in  \Rbb^m $.

If $\rho$ is a probability measure with mean $\bar u := \int u \rho(u) $ equal to zero, then $T$ is the covariance operator of $\rho$ and $M_n$ is the space of principal components of $\rho$. If $\rho$  has  mean $\bar u \neq 0 ,$ we can consider the covariance  operator $T v =\int_M (u-\bar u) (u- \bar u , v)_X d\rho(u) $, and $u \in M$ is approximated by $\bar u + P_{M_n} (u-\bar u),$ with $M_n$ the dominant eigenspace of $T$.

\subsection{Nonlinear approximation}\label{sec:nonlin-approx}

For many practical applications, linear approximation methods present a bad performance, which requires to introduce   nonlinear decoders. 

For any (relatively) compact $M$, letting  $\Ec_n$ and $\Dc_n$ be the sets of all possible nonlinear maps yields 
 $w(M , \Ec_n,\Dc_n)_X = 0$ for any $n\ge 1.$ However, this corresponds  to unreasonable approximation methods.  
Restricting both the  decoders and encoders to be continuous yields the notion of nonlinear manifold width of DeVore-Howard-Micchelli \cite{DeVore1989}
$$
\delta_n(M)_X = \inf_{D\in C(\Rbb^n ; X ) , E \in  C(  X ; \Rbb^n) }  \sup_{u \in M} \Vert u- D\circ E(u) \Vert_X,$$
which represents the  performance of an optimal continuous approximation process. For numerical stability reasons, continuity is in general not sufficient. 
Further assuming that encoders and decoders are Lipschitz continuous yields the notion of stable width introduced in \cite{Cohen2022Jun}. Lipschitz continuity ensures the stability  of the corresponding approximation process with respect to perturbations, a crucial property when it comes to practical implementation.  
However, even when imposing stability,  nonlinear widths are usually associated with optimal nonlinear encoders whose practical implementation may be difficult or even infeasible, e.g. associated with NP-hard optimization problems. This difficulty appears when $D(a)$ is a neural network or a tensor network with parameters $a$, and we aim at finding an optimal parameter value $a$. 

Restricting the encoder to be linear and continuous and allowing for arbitrary nonlinear decoders yields sensing numbers 
$$
s_n(M)_X = \inf_{D : \Rbb^n \to X} \inf_{\ell_1,\hdots, \ell_n}  \sup_{u \in M} \Vert u- D(\ell_1(u), \hdots, \ell_n(u)) \Vert_X,
$$
where the infimum is taken over all linear forms $\ell_1, \hdots,\ell_n$ and all nonlinear maps $D$. This provides a benchmark for all nonlinear approximation methods with linear encoders, which is relevant in many applications where the available information on $u$ is linear (point evaluations of functions, local averages of functions or  more general linear functionals).  In information-based complexity (IBC) \cite{Traub1998}, $s_n(M)_X$ is known as the worst-case error of non-adaptive deterministic algorithms using $n$ linear  information. Similar notions of error exist for average case setting, and random information or algorithms.   

\subsection{Nonlinear approximation with linear encoders and nonlinear decoders with values in linear spaces}

The class of nonlinear decoders has to be restricted to some subsets of decoders with feasible implementation. 
A practical approach consists in restricting the encoder to be linear and the decoder $D$ to take values in some linear  space $X_N$ with dimension $N \ge n$ \cite{Barnett_2022, Geelen_2023, Barnett2023Nov, geelen2024}. The range of $D$ is a nonlinear manifold $M_n$ in $X_N$. 
Given a basis $\varphi_1,\hdots, \varphi_N$ of $X_N$, the decoder has the form 
\begin{align}
D(a) = D_L(a) + D_N(a),
\label{decoderform}
\end{align}
where $D_L$ is the  linear  operator from $\Rbb^n$ to $X_n := \mathrm{span}\{\varphi_1, \hdots,\varphi_n\}$ defined by
$$
D_L(a) = \sum_{i=1}^n a_i \varphi_i,
$$
and $D_N$ maps $\Rbb^n$ to the complementary space of $X_n$ in $X_N$, 
$$
D_N(a) = \sum_{i=n+1}^{N} g_i(a) \varphi_i,
$$ 
where the functions $ g_i : \mathbb R^n \rightarrow \mathbb R $ are nonlinear maps. \blue{Note that the functions $\varphi_1,\dots,\varphi_n$ are related to the encoder}.

 For an orthonormal basis, if we choose the linear encoder $E(u) = (a_i(u))_{i=1}^n$ with $a_i(u) = (u,\varphi_i)_X$, then  $ D_L (E(u)) = P_{X_n} u$ is the orthogonal projection of $u$ onto $X_n$, and $D_N(E(u))$ is a nonlinear correction in the orthogonal complement of $X_n$ in $X_N.$ 
The space $X_n$ can be seen as a choice of a coordinate system   for  the description of the nonlinear manifold $M_n$.
\red{The space $X_N$ can be optimized but a practical choice  consists in taking for this space the optimal or near-optimal space of linear methods}, given by principal component analysis (optimal in mean-squared error) or greedy algorithms (close to optimal in worst case error), see Section~\ref{sec:lin-approx}. 

\begin{remark}
We can also consider a decoder with values in an affine space $\bar u + X_N$, and an affine encoder providing the coefficients $a = E(u)$ of $u- \bar u$ in the subspace $X_n \subset X_N$. This is classical when performing a principal component analysis, where $\bar u = \int_M u \, d\rho(u)$ is chosen as the expectation of a probability  measure $\rho$ on $M$.
\end{remark}

\begin{remark}
Even if the decoder is nonlinear, the error is  lower bounded by the Kolmogorov width  $ d_N(M)_X$ in  worst-case setting or by the average Kolmogorov width $d^{(2)}_N(M,\rho)_X$ in  mean-squared setting. Hence, when these widths present a slow convergence  with $N$, the approach may require a large value of $N$.
\end{remark}

The above structure of the decoder is based on the observation that in many applications, for an element $u\in M$, the coefficients $a_i(u)$ for $i>n$ can be well approximated as functions $g_i(E(u))$ of a few coefficients $E(u)= (a_i(u))_{i=1}^n$.
However, even for manifolds $M$ that can be well approximated by such a low-dimensional manifold $M_n$, the functions $g_i$, and therefore the map $D_N$ may be highly nonlinear and difficult to estimate. The authors in \cite{Barnett_2022,Geelen_2023} restrict the maps $g_i$ to quadratic polynomials. Higher degree polynomials were used in \cite{geelen2024}, but only sums of univariate polynomials were considered, for complexity issues. Such restrictions may lead to  a limited performance of these methods in practice. 
The authors in \cite{Barnett2023Nov,cohen2023}  used highly expressive approximation tools such as neural networks or random forests, but the resulting accuracy is not what could be expected from these tools. This is due to the difficulty of learning with such approximation tools using limited data. 
\\
\par 

In this paper, we propose a new decoder architecture based on compositions of polynomials. \blue{The architecture is based on the following observation made through our own numerical experiments, and which we illustrate in section \ref{sec:experiments}. \textit{A coefficient $a_i(u)$ for $i>n$ may have a highly nonlinear  relation with the first $n$ coefficients $a = E(u)$ but a much smoother relation when expressed in terms of $a$ and additional coefficients $a_j(u)$ with $n<j<i$}}. This observation suggests the following compositional structure of the decoder's functions 
$$
g_i(a) = f_i(a, (g_{j}(a))_{n<j\le n_i}),
$$ 
where \blue{$n_i < i$ and} the $f_i$ are polynomial functions. 
The use of global polynomial functions allows us to learn functions $f_i$ from  a finite training sample, although the variables $(a , (g_j(a))_{n<j\le n_i})$ take values in a set of measure zero in $\Rbb^{n_i}$.

\begin{example}
 As a simple and illustrative example, consider a one-dimensional manifold $M = \{ u(t)= a_1(t) \varphi_1 + a_2(t) \varphi_2 + a_3(t) \varphi_3 
 : t\in [-1,1] \}$, with $a_1(t) = t$, $a_2(t) = 5t^3 - 4t $ and $a_3(t) = 625 t^{10} - 1500 t^8 + 1200 t^6 - 340 t^4 + 16 t^2$. 
 $M$ is contained in the $1$-dimensional manifold  
 $M_1 = \{ D(a_1) :=  a_1 \varphi_1 + g_2(a_1) \varphi_2  + g_3(a_1) \varphi_3 : a_1 \in \Rbb \}$, with  $g_2(t) = a_2(t)$ and $g_3(t) = a_3(t).$ 
 Figures \ref{fig:toy_manifold_a2_a1} and \ref{fig:toy_manifold_a3_a1} represent the functional relation between $a_2$ and $a_1$ and between $a_3$ and $a_1$, respectively.  On Figure \ref{fig:toy_manifold_f3}, we observe the manifold  $M$ and a bivariate polynomial $f_3$ \blue{which is smoother than the univariate polynomial $g_3$} and is such that $a_3(t) = f_3(a_1(t) , a_2(t)) = f_3(a_1(t), g_2(a_1(t)))$.
 

\begin{figure}[h]
    \centering
    \subfigure[$ a_2 $ as function of $ a_1 $\label{fig:toy_manifold_a2_a1}]{\includegraphics[width=0.33\textwidth]{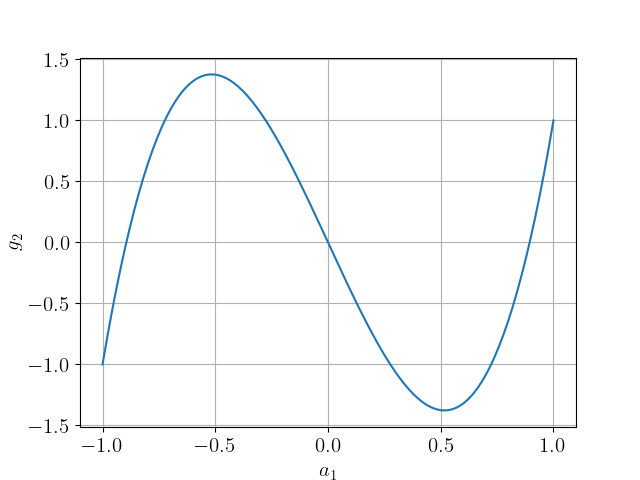}} 
    \subfigure[$ a_3 $ as function of $ a_1 $\label{fig:toy_manifold_a3_a1}]{\includegraphics[width=0.33\textwidth]{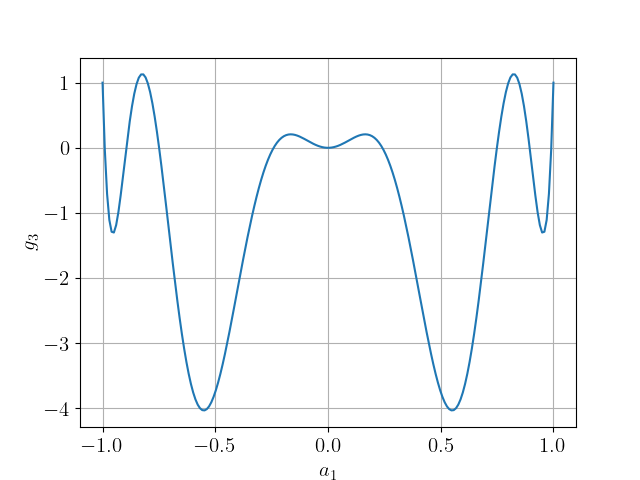}}
    \subfigure[Manifold $M$  (black curve) and bivariate polynomial function $f_3$ such that $a_3 = f_3(a_1,a_2)$ (surface); \blue{$f_3$ is smoother than $g_3$}\label{fig:toy_manifold_f3}]{\includegraphics[width=0.33\textwidth]{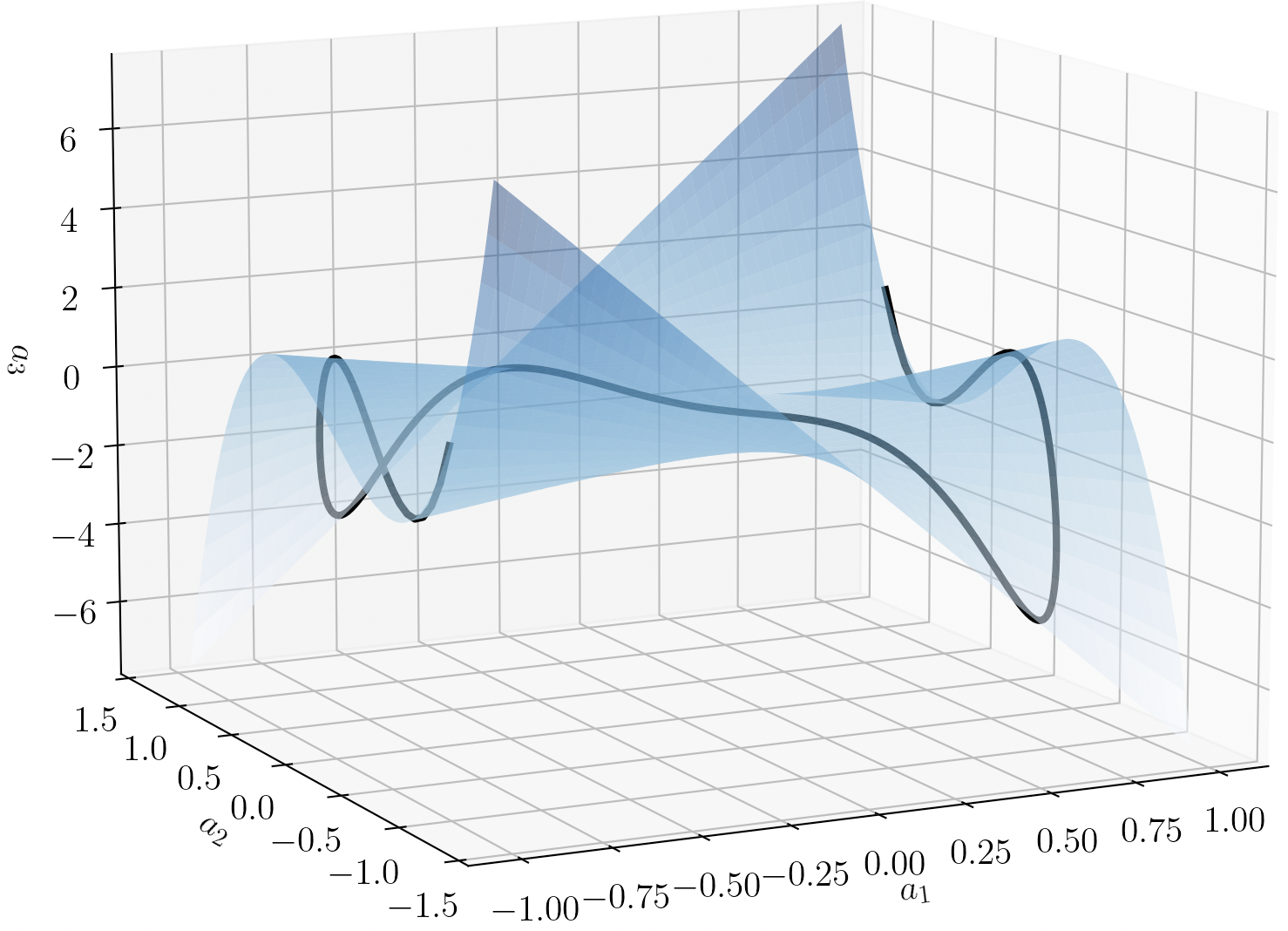}}
    \caption{Illustrative example for a one-dimensional manifold in $\Rbb^3$.}
\label{fig:toy_manifold}
\end{figure}
\end{example}

\blue{In the references cited above, the functions $\varphi_1, \dots, \varphi_n $ are usually chosen such that $X_n$ is an optimal (or near optimal) space for linear approximation. However, this choice is in general not optimal for nonlinear approximation, as demonstrated in  \cite{schwerdtner2024greedyconstructionquadraticmanifolds} for quadratic manifold approximation, where the authors propose a greedy algorithm for selecting an index set $I=\{1, \ldots, n\}$ and the corresponding space $X_n$ generated by the functions $\{\varphi_i : i\in I\}$. We next follow the same idea, with an algorithm selecting adaptively an index set $I$.}

 \section{Nonlinear approximation using  compositional polynomial networks}\label{sec:method}

In this section, we introduce a manifold approximation method based on a linear/affine encoder and a  nonlinear decoder based on compositions of polynomials. We provide error and stability analyses that will allow us to design a controlled constructive algorithm in the next section. 

\subsection{Description of the encoder and decoder}
Let $\bar u\in X$ and $\varphi_1, \hdots, \varphi_N$ be an orthonormal basis of a space  $X_N$, e.g. provided by some optimal or near to optimal linear approximation method, and let $X_n$ be a subspace of $X_N$ spanned by basis functions $\varphi_i$, $i\in I = \{i_1, \hdots,i_n\}$.
We consider for the encoder $E$
the map which associates to $u \in M$  the coefficients $ E(u) = ((u-\bar u,\varphi_i)_X)_{i \in I}$ of the orthogonal projection of $u-\bar u$ onto $X_n$. For the decoder, we consider the map $ D : \ell_2(I) \rightarrow \blue{X}$ defined by
$$
D(a) = \bar u +  \sum_{i = 1}^N g_i(a) \varphi_i
$$
with $g_{i_k}(a) := a_k$ for $k=1,\hdots,n$, and for all $i \in I^c := \{1,\hdots, N\} \setminus I$, 
$$
g_i(a) = f_i( (g_j(a))_{j\in S_i}), 
$$  
with  $S_i = \{1 , \hdots, n_i\}$, $n_i < i$ and  $f_i : \Rbb^{n_i} \to \Rbb$ a $n_i$-variate polynomial whose structure will be discussed later. \blue{Note that the choice of $S_i \subset \{1,\dots,i-1\}$ results in a  feed-forward compositional structure for $D$.}

 The choice of the set $I$ should be  determined  such  that $\bar u + X_n$ provides a good approximation of $M$, but also such that the functions $g_i$ have a low complexity. 
The ordering of basis functions may be optimal \blue{for linear approximation} in the sense that the set $I = \{1, \hdots,n\}$ provides an optimal linear subspace $X_n$ of dimension $n$ for linear approximation, for any $1\le n \le N$.  \blue{However, as mentioned earlier, this does not necessarily result in an optimal space for nonlinear approximation.} It may be relevant to define an encoder associated with indices $I$ that cannot be well approximated as functions of other coefficients. This is the reason why we here consider a general set $I$, whose practical construction will be discussed in Section \ref{sec:algorithm}.

\subsection{Control of approximation error}\label{sec:error}

We here analyse the error of approximation of the manifold $M$ by $D\circ E(M)$, in mean-squared or worst case settings. 
This is equivalent to measuring the quality of approximation of the identity map $id : M \to X$ by $D\circ E$. 
Consider the mean-squared error
$$
e_2(D\circ E) := \Vert id - D\circ E \Vert_2 := \left( \int_M \Vert u - D(E(u)) \Vert_X^2 d\rho(u)  \right)^{1/2}
$$
and the worst-case error
$$
e_\infty(D\circ E) := \Vert id - D\circ E \Vert_\infty := \sup_{u\in M} \Vert u - D(E(u)) \Vert_X.
$$
\blue{In the mean-squared setting, we assume that $\rho$ has finite second order moment, i.e. $e_2(0)^2 = \int_M \Vert u\Vert_X^2 d\rho(u) <\infty$. In the worst-case setting, we assume that $M$ is bounded, i.e. $e_\infty(0) := \sup_{u\in M} \Vert u\Vert_X <\infty$.} 
We aim at providing a methodology which guarantees a prescribed  relative precision $\epsilon$, i.e.
$$
e_p(D\circ E) \le \epsilon \; e_p(0), 
$$
with $p=2$ or $p =  \infty$.  
An element $u \in M$ admits a decomposition 
$$
u = \bar u + \sum_{i=1}^N a_i(u) \varphi_i + r_N(u),
$$
with $a_i(u) = (\varphi_i , u-\bar u)_X$ and $r_N(u) = (id - P_{X_N}) (u-\bar u)$, and  
$$
\Vert u - D(E(u)) \Vert_X^2 = \sum_{i \in I^c} \vert a_i(u) - g_i(E(u)) \vert^2 + \Vert r_N(u) \Vert_X^2   .
$$
We deduce the following bounds of the mean-squared and worst-case errors.
\begin{lemma}
It holds 
$$
e_2(D\circ E)^2 = \sum_{i \in I^c} \epsilon_{i,2}^2 + \Vert r_N \Vert_2^2,
$$
with 
$$
\epsilon_{i,2}^2 = \Vert a_i - g_i \circ E \Vert_2^2  = \int_M \vert a_i(u) - g_i(E(u)) \vert^2 d\rho(u),
$$
and 
$$
e_\infty(D\circ E) \le \sum_{i \in I^c} \epsilon_{i,\infty} + \Vert r_N \Vert_\infty,
$$
with 
$$
\epsilon_{i,\infty}
= \Vert a_i - g_i \circ E \Vert_\infty
= \sup_{u\in M} \vert a_i(u) - g_i(E(u)) \vert.
$$
\end{lemma}

The above lemma shows that the error can be controlled with a suitable choice of $N$ and with a suitable control of the errors $\epsilon_{i,p}$. 

\begin{proposition}\label{prop:error-control}
Let $\epsilon>0$, $\beta \in [0,1]$ and $p=2$ or $p= \infty$. Let 
$N$ be the minimal integer  such that $$
\Vert r_N \Vert_p \le \beta \, \epsilon \,  e_p(0).
$$
Assume that for all $i\in I^c$, 
$$
\epsilon_{i,p} \le \bar{\epsilon}_{i,p},
$$
with 
$$
\sum_{i\in I^c} \bar{\epsilon}_{i,2}^{\,2} \le   (1 - \beta^2)  \, \epsilon^2 \, e_2(0)^2,
$$
for $p=2$, or 
$$
\sum_{i\in I^c} \bar{\epsilon}_{i,\infty} \le   (1 - \beta) \, \epsilon \, e_\infty(0),
$$
for $p=\infty$. Then it holds 
$$
e_p(D\circ E) \le   \epsilon\, e_p(0). 
$$ 
\end{proposition}

A natural choice is to equilibrate the error due to the projection onto $X_N$ and the errors due to the approximation  of coefficients, by taking 
 $\beta = 1/\sqrt{2}$ for $p=2$ or $\beta = 1/2$ for $p=\infty$. However, the parameter $\beta$ allows an additional flexibility. A high value of $\beta$ may be interesting when the convergence of $\Vert r_N \Vert_p$ with $N$ is slow and the coefficients $a_i(u) = (\varphi_i, u -\bar u)_X$ are rather easy to approximate in terms of $E(u)$. A low value of $\beta$ may be relevant when  $\Vert r_N \Vert_p$ decreases rapidly with $N$ and the coefficients $a_i(u)$ are difficult to approximate in terms of $E(u).$

The choice of an increasing sequence of tolerances $(\bar\epsilon_{i,p})_{i\in I^c}$ allows to require less and less precision for the approximation of coefficients $a_i(u)$ as the index $i$ increases. Indeed, when the sequence of functions $\varphi_i$ results from an optimal or near optimal linear approximation method, with a natural ordering, coefficients $a_i(u)$  are in general more and more difficult to approximate as the index $i$ increases. A practical choice consists in taking 
$$
\bar \epsilon_{i,2} = \omega_i^{1/2} (1 - \beta^2)^{1/2}  \, \epsilon \, e_2(0)\quad \text{or} \quad \bar \epsilon_{i,\infty} =  \omega_i (1 - \beta) \, \epsilon \, e_\infty(0),
$$
with weights $\omega_i > 0$ such that $\sum_{i\in I^c} \omega_i = 1$, e.g.
\begin{equation}
\omega_i = i^{\alpha} / \sum_{j\in I^c} j^\alpha, \quad i \in I^c, \label{eq:weights}
\end{equation}
for some $\alpha >0$.

The infimum of the error $\epsilon_{i,p}$ over all possible maps $g_i$ may not be zero. This is the case when there is no functional relation between $a_i(u)$ and the coefficients $E(u).$ Also, when using for the maps $f_i$ some approximation tool with limited complexity  (e.g. polynomials with some prescribed degree), the obtained error $\epsilon_{i,p}$ may not achieve the required precision. In these situations, the index $i$ should be taken in the set $I$ of reference parameters.     
This suggests an adaptive choice of $I$ and the associated encoder, as will be proposed in Section \ref{sec:algorithm}.  

\subsection{Control of stability}\label{sec:stability}

For practical use, it is important to control the stability of the approximation process, or sensitivity to perturbations. 
The chosen encoder, as a map from $X$ to $\ell_2(I)$, is $1$-Lipschitz and therefore ideally stable. Indeed, for any $u,\tilde u\in X$,
$$
\Vert E(u) - E(\tilde u) \Vert_{2} = (\sum_{i \in I} \vert (u - \tilde u,\varphi_i)_X \vert^2 )^{1/2} = \Vert P_{X_n}(u-\tilde u) \Vert_2 \le \Vert u-\tilde u \Vert_X.
$$

Let $ A = \{E(u) :  u \in M \} \subset  \ell_2(I)  $ be the set of parameters values when $ u $ spans the whole set $ M$. We study the Lipschitz continuity of the decoder as a map from $ A$ to $X$. Let $\gamma = (\gamma_i)_{i\in I^c}, \blue{\gamma_i > 0}$. For $i\in I^c$, we let 
$ B_i = \{b =  (g_j(a))_{j \in S_i} : a \in A \} \subset \Rbb^{n_i}$ which we equip with the norm 
$$
\Vert b \Vert_{i, \gamma} = \max\{ \Vert (b_j)_{j \in I \cap S_i} \Vert_2 , \max_{j \in S_i \cap I^c} \gamma_j^{-1} \vert b_j \vert\},
$$
and we  define a corresponding Lipschitz norm of a function $f_i:B_i \to \Rbb$ as 
\begin{equation}
\| f_i \|_{i, \gamma} = \max_{b,b'\in B_i} \frac{|f_i(b) - f_i(b') |}{\Vert b-b' \Vert_{i,\gamma} }. \label{eq:Lipi}
\end{equation}
\begin{proposition}\label{prop:lip-decoder}
Let $\gamma = (\gamma_i)_{i\in I^c}$ \blue{with $\gamma_i > 0$} and assume $\| f_i \|_{i, \gamma} \le \gamma_i$ for all $i\in I^c$. Then for all $a,a'\in A$, it holds 
$$
\Vert D(a) - D(a') \Vert_X \le (1+ \sum_{i\in I^c} \gamma_i^2)^{1/2} \Vert a - a'\Vert_2.
$$
\end{proposition}
\begin{proof}
For $ a, a' \in A $, it holds 
$$
\|D(a) - D(a') \|_X^2 = \Vert a - a'\Vert_2^2 + \sum_{i \in I^c }^N |g_i(a) - g_i(a') |^2.
$$
We then show by induction that functions $g_i$ are $\gamma_i$-Lipschitz from $\ell_2(I)$ to $\Rbb$, for all $i\in I^c$. For $i$ the minimal integer in $I^c$,  \blue{we easily check that $S_{i} \subseteq I$} and therefore 
\begin{align*}
 | g_i(a) - g_i(a') | &= | f_i( (g_j(a))_{j\in S_i} ) - f_i( (g_j(a'))_{j\in S_i}) | \\
 &\le \gamma_i \Vert (g_j(a))_{j \in S_i} - (g_j(a'))_{j \in S_i} \Vert_2 \\
 & \leq \gamma_i \Vert a - a'\Vert_2.
\end{align*}
Then for any  $i \in I^c$, assuming that $g_j$ is $\gamma_j$-Lipschitz for $j \in S_i \subset \{1, \hdots, i-1\}$,  it holds 
\begin{align*}
    | g_i(a) - g_i(a') | 
    & \le \gamma_i \max\{ \Vert (g_j(a) - g_j(a'))_{j \in S_i \cap I} \Vert_2 , \max_{j \in S_i \cap I^c} \gamma_j^{-1} \vert g_j(a) - g_j(a') \vert\}\\
    &\le \gamma_i \max\{ \Vert a-a' \Vert_2 , \max_{j \in S_i \cap I^c} \gamma_j^{-1} \gamma_j \Vert a-a' \Vert_2 \}\\
    &= \gamma_i  \Vert a-a' \Vert_2.
\end{align*}
This concludes the proof.
\end{proof}
\begin{remark}[Estimation of Lipschitz norms]
In practice, we can  estimate the Lipschitz constant $ \| f_i \|_{i,\gamma} $ given by \eqref{eq:Lipi} using 
 training data. More precisely, given samples $u^{(1)}, \hdots, u^{(m)}$ in $M$ and corresponding samples $a^{(k)} = E(u^{(k)}) \in A$ and $b^{(k)} = (g_j(a^{(k)}))_{j\in S_i} \in B_i$, $1\le k \le m$, we estimate  
$$
\| f_i \|_{i,\gamma} \approx \max_{1 \le k < k' \le m} \frac{|f_i(b^{(k)}) - f_i(b^{(k')}) |}{\Vert b^{(k)}-b^{(k')} \Vert_{i,\gamma} }.
$$
\end{remark}

\section{An adaptive algorithm}\label{sec:algorithm}

We now present an adaptive algorithm for the construction of the encoder and decoder, in order to obtain a prescribed   precision for the approximation of a manifold, in  mean-squared or worst-case settings, and a control of the stability of the decoder.  

Our objective is to provide an algorithm that delivers an encoder $E$ and a decoder $D$ such that 
$$
e_p(D\circ E) = \Vert id - D\circ E \Vert_p \le \epsilon \; e_p(0),
$$
with $p=2$ or $p=\infty$, 
and such that 
$$
\Vert D(a) - D(a') \Vert_2 \le L \Vert a-a'\Vert_2
$$
for all $a,a' \in E(M),$ where $\epsilon$ and $L$ are prescribed by the user. 
Proposition \ref{prop:error-control} provides conditions on the truncation index $N$ and    the approximation errors $\epsilon_{i,p}$ of the coefficients by the maps $g_i$ in order to satisfy the relative precision $\epsilon. $  Proposition \ref{prop:lip-decoder}  
  provides a condition on the Lipschitz norms of functions $f_i$ in order to guarantee that $D$ is $L$-Lipschitz. In practice, we rely on a finite number of samples $u^{(1)}, \hdots, u^{(m)}$ in $M$, which is equivalent to consider the above conditions with a discrete manifold $M = \{u^{(1)}, \hdots, u^{(m)}\}$ or a discrete measure $\rho = \sum_{i=1}^m \delta_{u^{(i)}}.$ Therefore, we keep a general presentation of the methodology. 

\subsection{Selection of the space $X_N$}

We determine the space $X_N$ from an optimal or near-optimal linear approximation method. 

\paragraph{Mean-squared setting}
For $p=2$, we can set $\bar u = 0$ and consider the  operator $T:X\to X$ defined by $T v = \int_M u (u,v)_X d\rho(u) $ and compute its eigenvectors $(\varphi_i)_{i\ge 1}$ sorted by decreasing eigenvalues $(\lambda_i)_{i\ge 1}$. The optimal  linear space $X_N$ of dimension $N$ is provided by the span of the first $N$ eigenvectors, and the error $\Vert r_N \Vert_2^2 \le \sum_{i>N} \lambda_i$. Therefore, $N$ is chosen as the minimal integer such that $ \sqrt{\sum_{i>N} \lambda_i} \le \beta \epsilon e_2(0).$   When $X = \Rbb^M$ and $\rho$ is a discrete measure $\rho = \sum_{i=1}^m \delta_{u^{(i)}}$, $T$ is identified with a matrix $T = \sum_{i=1}^m u^{(i)} \otimes u^{(i)} $. When $\Rbb^M$ is equipped with the Euclidian norm $\Vert \cdot\Vert_2$, the vectors $\varphi_i$ are obtained as the $M$ dominant left singular vectors of the matrix $A = (u^{(1)} , \hdots, u^{(m)}) \in \Rbb^{M\times m}$, with associated singular values $\sigma_i$ such that $\lambda_i = \sigma_i^2.$ When $\Rbb^M$ is equipped with a norm $\Vert v \Vert_X^2 = v^T M_X v = \Vert M_X^{1/2} v \Vert_2^2$, with $M_X$ a symmetric positive definite matrix, the $(\varphi_i , \sigma_i)$ are the left singular vectors and associated singular values of the matrix $M_X^{1/2} A$.
 
 When $\rho$ is a probability measure, we can choose $\bar u = \int_M u d\rho(u)$ and consider instead the covariance  operator $T v = \int_M (u-\bar u)(u-\bar u,v)_X d\rho(u) $ or in the discrete setting, the matrix $T = \sum_{i=1}^m (u^{(i)}-\bar u) \otimes (u^{(i)} - \bar u)$ or the corresponding matrix $A = (u^{(1)}-\bar u, \hdots , u^{(m)}-\bar u)$.
 
 \begin{remark}
For high-dimensional problems ($M\gg 1$), we can rely on randomized methods for estimating near-optimal linear spaces $X_N$, see \cite{Halko2011May} or \cite{balabanov2019randomized} in the context of model order reduction. 
 \end{remark}
 
\paragraph{Worst-case setting}
For $p=\infty$, we can use a (weak) greedy algorithm to determine a near optimal space $X_N = \mathrm{span}\{v_1, \hdots,v_N\}$, as described in Section \ref{sec:lin-approx}, and obtain the vectors $\varphi_1, \hdots, \varphi_N$ by Gram-Schmidt orthonormalization. Provided an   upper bound  $\delta_N(u)$ of $\Vert r_N(u) \Vert_X = \Vert u - P_{X_N} u \Vert_X$, we run the greedy algorithm and select $N$ as the minimal integer such that 
$\sup_{u\in M} \delta_N(u) \le \epsilon \, e_\infty(0) = \epsilon \, \sup_{u\in M} \Vert u \Vert_X$. In a discrete setting, where we 
only have access to samples $u^{(1)}, \hdots, u^{(m)}$ in $M$, we replace $M$ in the above expression by the discrete set $M = \{u^{(1)}, \hdots, u^{(m)}\}$.

\subsection{Control of the error and stability}
Satisfying a prescribed precision for the errors $\epsilon_{i,p}$ or a prescribed bound for Lipschitz constants of maps $f_i$ may be a difficult task for some indices $i\in \{1, \hdots, N\}$. This may require to progressively adapt the set of indices $I$ associated with the encoder, whose coefficients are not approximated, and the subsets of parameters $S_i$ for $i\in I^c$. At one step of the adaptive algorithm, we are given the set $I$ and its complementary set $I^c$ in $\{1,\hdots, N\}$. Some coefficients, with indices denoted by $J \subset I^c$, remain to be approximated with a controlled precision and stability, while  coefficients with indices $I^c \setminus J$  have been already   approximated with success at previous steps. 
At this step, for the approximation of coefficients with indices $J$, we demand that the errors and Lipschitz constants  satisfy 
$$\epsilon_{i,p} \le \bar \epsilon_{i,p} \quad \text{and} \quad   \gamma_i = \Vert f_i \Vert_{i,\gamma} \le \bar \gamma_i, \quad \text{for all } i\in J, $$ 
with suitable bounds $\bar \epsilon_{i,p}$ and $\bar \gamma_i$ defined below.  At the end of the step, if some indices $i\in J$ do not satisfy the above error and stability conditions, we augment the corresponding sets  $S_i$ and eventually  augment the set $I$, as described in     section \ref{sec:description-algo}.

\paragraph{Control of error for $p=2$}
In the mean-squared setting ($p=2$), we want to satisfy 
$$
\sum_{j \in I^c \setminus J} \epsilon_{j,2}^2 +   \sum_{j \in J} \epsilon_{j,2}^2  \le (1- \beta^2) \epsilon^2 e_2(0)^2 .
$$
Therefore, we can define for all $i \in J$
\begin{equation}
    \bar{\epsilon}_{i,2}^2:=   \omega_i ( (1- \beta^2) \epsilon^2 e_2(0)^2 - \sum_{j \in I^c \setminus J} \epsilon_{j,2}^2) \label{bound-bar-eps2}
\end{equation} 
with weights $(\omega_i)_{i \in J}$ such that $\sum_{i\in J}   \omega_i =1$, e.g. $ \omega_i = i^{\alpha} / \sum_{j\in J} j^\alpha$ with $\alpha \ge 0$. The weights $\omega_i$ of the remaining indices $i\in J$ are updated at each step of the algorithm. In the above expression, $\bar \epsilon_{i,2}$ involves the true errors $\epsilon_{j,2}$ for indices $j\in I^c \setminus J$ that have been already approximated with success. A more conservative choice consists in defining  
\begin{equation}
    \bar{\epsilon}_{i,2}^2:=   \omega_i (1- \beta^2) \epsilon^2 e_2(0)^2,
    \label{bound-bar-eps2-conservative}
\end{equation} 
where the weights $(\omega_i)_{i \in J}$ at this step are chosen such that $
\sum_{i \in J} \omega_i = 1 - \sum_{j \in I^c \setminus J} \omega_j$. The weights $\omega_j$, $j\in I^c\setminus J$, correspond to the values determined at the steps where the corresponding coefficients have been approximated with success.
We define  $\omega_i := \omega'_j \big(1 - \sum_{j \in I^c \setminus J} \omega_j \big)$, with $ \sum_{j \in J} \omega'_i = 1$, e.g $ \omega'_i := i^\alpha / \sum_{j \in J} j^\alpha$, $\alpha \geq 0$.
 
\paragraph{Control of error for $p=\infty$}
In the worst case setting ($p=\infty$), we want to satisfy 
$$
    \sum_{j \in I^c \setminus J} \epsilon_{j,\infty} +   \sum_{j \in J} \epsilon_{j,\infty}  \le (1- \beta) \epsilon\, e_\infty(0).
$$
Following the same reasoning as above,   
we can define for each $i\in J$  
\begin{equation} \bar{\epsilon}_{i,\infty}:=  \omega_i ( (1- \beta) \epsilon e_\infty(0) - \sum_{j \in I^c \setminus J} \epsilon_{j,\infty})
\label{bound-bar-epsinf}
\end{equation} 
 with  weights $ (w_i)_{i\in J}$ such that $\sum_{i\in J}  w_i =1$, or for  a more conservative condition, we can define 
\begin{equation}\ \bar{\epsilon}_{i,\infty}:=  \omega'_i ( 1 - \sum_{j \in I^c \setminus J} \omega_j) (1- \beta) \epsilon \, e_\infty(0), 
 \label{bound-bar-epsinf-conservative}
\end{equation} 
where $ \omega'_i := \sum_{j \in J} \omega'_i=1$, for e.g $ \omega'_i := i^\alpha / \sum_{j \in J} j^\alpha$, $\alpha \geq 0$. 

\paragraph{Control of stability}

In order to guarantee that the decoder $D$ is $L$-Lipchitz, we want to satisfy 
$$
1 +  \sum_{j\in I^c\setminus J} \gamma_j^2+  \sum_{j\in J} \gamma_j^2 \le L^2 .
$$  
Therefore, we can define the values of $\bar \gamma_i$ for $i\in J$ as  
\begin{equation}
\bar \gamma_i^2 := \tilde \omega_i (L^2-1 - \sum_{j\in I^c\setminus J} \gamma_j^2),
\label{bound-gamma}
\end{equation}
with $\sum_{i\in J} \tilde \omega_i = 1$, e.g. $ \tilde \omega_i = i^{\alpha} / \sum_{j\in J} j^\alpha$ with $\alpha\ge 0$. The 
 $ \gamma_j$ for $j\in I^c\setminus J $ are the Lipschitz norms $\Vert f_j \Vert_{j, \gamma}$ of the functions $f_j$ that have been determined with success at previous steps of the algorithm. For a more conservative choice, we can set 
\begin{equation}
\bar \gamma_i^2 := \tilde \omega'_i ( 1 - \sum_{j\in I^c\setminus J} \tilde \omega_j) (L^2-1),
\label{bound-gamma-conservative}
\end{equation}
still   with $\sum_{i\in J} \tilde \omega'_i = 1$, and where the $\tilde \omega_j$, $j\in I^c\setminus J$, correspond to the values determined at the steps where the corresponding coefficients have been approximated with success.

 \subsection{Description of the algorithm}\label{sec:description-algo}

Based on the above analysis, we propose the following adaptive algorithm. 

 We first set $I = \{1, \hdots, n\}$ for some $n\ge 1.$ The choice of $n$ may be guided by an a priori estimation of the dimension of the manifold $M$ (e.g. if generated by some map from $\Rbb^n$ to $X$), or it can be determined such that   
  $\Vert id -   P_{X_n} \Vert_p \le  \epsilon_0 \; e_p(0) $  with $\epsilon_0 $  larger than the desired precision $\epsilon$, but sufficiently small to guarantee that the approximation $\bar u + P_{X_n} (u-\bar u)$ provides a reasonable approximation of $u$ for all $u \in M$. 
  
 We then let $J = \{n+1, \hdots, N\}$ be the indices of the coefficients to be approximated. 
We set $S_i = I$ for all $i\in J$ and set $k=n$. 
Then while $J\neq \emptyset$, we perform the following steps 
\begin{itemize}
\item Update the bounds  $\bar \epsilon_{i,p}$ for the errors $\epsilon_{i,p}$, for all $i\in J$, using \eqref{bound-bar-eps2} or  \eqref{bound-bar-eps2-conservative} for $p=2$, or \eqref{bound-bar-epsinf} or  \eqref{bound-bar-epsinf-conservative} for $p=\infty$.
\item Update the bounds ${\bar{\gamma}}_i$ for the Lipschitz constants $\Vert f_i \Vert_{i,\gamma}$, for all $i\in J$, using \eqref{bound-gamma} or \eqref{bound-gamma-conservative}. 
\item For all $i\in J$,
\begin{itemize}
\item Compute a polynomial $f_i$ (see Section \ref{sec:approx-poly}),
\item Compute the error $\epsilon_{i,p} = \Vert a_i - g_i\circ E \Vert_p$ and Lipschitz norm
 $ \gamma_i = \Vert f_i \Vert_{i,\gamma}$,
\item If $\epsilon_{i,p} \le \bar{\epsilon}_i$ and $\gamma_i \le  {\bar{\gamma}}_i$, then 
set $J \leftarrow J \setminus \{i\}$, otherwise 
set $S_i \leftarrow S_i \cup \{k+1\}.$
\end{itemize}
\item If $k+1 \in J$, then set $I \leftarrow I \cup \{k+1\}$ and  $J \leftarrow J \setminus \{k+1\}$.
\item Set $k\leftarrow k+1$.
\end{itemize}

\subsection{Polynomial approximation}\label{sec:approx-poly}

In this section, for a given $i \in I^c,$ we discuss the approximation of the coefficient $a_i(u) := y$ by a polynomial  $f_i(x)$ 
of variables $x := (g_j(a(u)))_{j \in S_i} \in \Rbb^{d}$ with $d:= \vert S_i \vert $.
Note that $y$ may not be a function of $x$ but may be approximated  sufficiently well by a function $f_i(x)$ in order to reach the target precision $\bar \epsilon_{i,p}$.  

In practice, we have access to samples $y^{(k)} = a_i(u^{(k)}) $ and  $x^{(k)} = (g_j(a(u^{(k)})))_{j=1}^d$, $1\le k \le m.$ We can then   estimate the function $f_i$ from the samples $\{(x^{(k)} , y^{(k)}) : 1\le k \le m\}$, e.g. by least-squares minimization 
$$
\min_{f_i \in \mathcal{P}} \, \sum_{k=1}^m \vert f_i(x^{(k)}) - y^{(k)} \vert^2 
$$
with $\mathcal{P}$ some set of polynomials over $\Rbb^d$, eventually constructed   adaptively.  The variable $x$ and the corresponding samples $x^{(k)}$    belong to the set $B := \{(g_j(a(u)))_{j \in S_i} : u\in M \}$ which may be a non trivial domain of $\Rbb^d$ but more importantly, $B$ is a set of Lebesgue measure zero in $\Rbb^d$ for $d> n.$  However,  approximation over such sets, even for a finite number  of samples, is still feasible with polynomials.

\begin{remark}
With this approach, we have a direct control of the error $\epsilon_{i,p}= \Vert a_i - g_i(a) \Vert_p = \Vert  y - f_i(x) \Vert_p$ with $x = (g_j(a))_{j\in S_i}$, or more precisely of their empirical counterparts. However, in a mean-squared setting, it is possible to rely on (cross-)validation methods to estimate the error $\epsilon_{i,p}.$ Another approach would consist in estimating $f_i$ from  samples $x^{(k)} = (a_j(u^{(k)}))_{j\in S_i}$ of the true coefficients $a_j(u) = (\varphi_i , u )_X$. The interest is that the estimation of the functions  $f_i$ can in principle be parallelized, since it does not require the knowledge of functions $g_j$, $j\in S_i$.   Also, there is a possible benefit since variables $x$ may now belong to a set with nonzero  Lebesgue measure, that could make the estimation of $f_i$ more robust. However, the control of errors and stability   requires the knowledge of the functions  $g_j$, $j\in S_i$, that makes parallelization not possible. This approach is not further considered in this paper. 
\end{remark}

Since the dimension $d$ is possibly large, structured polynomial approximations are required. 

\paragraph{Sparse approximations}

One standard approach is to use sparse polynomial approximation by considering 
$$
\mathcal{P} = \Pbb_\Lambda := \mathrm{span}\{x^\lambda := x_1^{\lambda} \hdots x_d^{\lambda_d} : \lambda \in \Lambda\},
$$
with $\Lambda$ a structured set of multi-indices in $\Nbb^d.$ 
A set $\Lambda$ is downward closed if for each $\lambda \in \Lambda$, all $\beta\in \Nbb^d$ such that $\beta \le \lambda$ belong to $\Lambda.$ For a downward closed set $\Lambda$, it holds  
$$
\Pbb_\Lambda  = \mathrm{span}\{\phi_\lambda(x) := \phi_{\lambda_1}^{1}(x_1)  \hdots \phi_{\lambda_d}^{d}(x_d) : \lambda \in \Lambda\}
$$
whatever the choice of univariate polynomial bases $\{\phi_{k}^{j}\}_{k\ge 0}$, where $\phi_{k}^{j}$ is a polynomial of degree $k$ in the variable $x_j$. 
Classical sets of multi-indices in high dimension include 
\begin{itemize}
\item Set with bounded total degree $p$: $\Lambda = \{\lambda\in \Nbb^d : \sum_{j=1}^d \lambda_j \le p\}$, with $\Lambda = \frac{(p+d)!}{p!d!}$,
\item Set with bounded partial degree $p$ and order of interaction $l$: $\Lambda = \{\lambda\in \Nbb^d : \lambda \le p , \vert \{j : \lambda_j \neq 0\} \vert \le l\}$,
\item Hyperbolic cross set with degree $p$:  $\Lambda = \{ \lambda \in \Nbb^d  \, : \, \prod_{j=1}^d (\lambda_j + 1) \leq p+1  \}$,  with $ | \Lambda | \sim p \text{log}(p+1)^d $.
\end{itemize}
The approximation can be further sparsified using a greedy algorithm or a sparsity-inducing regularization (e.g. $\ell_1$ regularization and LARS-homotopy algorithm), which yields a sequence of subsets of $\Lambda$, one of which being selected using model selection method (e.g., using cross-validation error estimation).   

\paragraph{Low-rank approximations}
An alternative approach to sparse polynomial approximation is to consider hierarchical low-rank approximations (or tree tensor networks), that consists in taking for $\mathcal{P}$ a set of low-rank tensors in the tensor product polynomial space $ \Pbb_{p} \otimes \hdots \otimes \Pbb_{p} = \Pbb_{p}^{\otimes d}.$ These approximations can be estimated using dedicated learning algorithms, see \cite{grelier:2018,michel2022}.

\paragraph{Choice of bases}
 For the stability of  least-squares approximation, it is important to work with orthonormal or near to orthonormal polynomial bases. 
The measure $\nu$ on $B$ is the push-forward measure of the measure $\rho$ on $M$ through the map which associates to $u\in M$ the vector $(g_j(E(u))_{j\in S_i}$. In practice, we only have access to samples $x^{(k)}$ of this measure. Therefore,  an orthonormal polynomial basis of $L^2_\nu(B)$ is out of reach in practice.  
We here propose to consider the domain $\Gamma = \Gamma_1 \times \hdots \times \Gamma_d \supset B$ such that $\Gamma_j = [\min_{u\in M} g_j(E(u)) , \max_{u\in M} g_j(E(u))]$, which can be estimated with samples $x_j^{(k)}$ of $g_j(E(u^{(k)}))$. Then we equip $\Gamma$ with the uniform measure and consider associated (tensor products of) Legendre polynomials. 
\begin{remark}
Tensor products of univariate Chebychev polynomials could be considered as well, that could be more relevant for a control of error   in uniform norm. We could also estimate the marginals $\nu_j$ of $\nu$ from samples $x_{j}^{(k)}$, and construct orthonormal polynomials with respect to these estimated measures. This may further improve the stability of least-squares approximation.
\end{remark}

\section{Numerical experiments}\label{sec:experiments}



We now illustrate the performance of the proposed method on three different benchmarks. In all benchmarks, we generate a  set of $m$ training samples $u_1,\dots,u_m$ in $X = \mathbb{R}^{D}$, which we equip with the euclidian norm.

For compositional polynomial networks (CPN), we use compositions of  sparse polynomials (CPN-S) or low-rank polynomials using tree tensor networks (CPN-LR). We choose $\bar u$, $X_n$ and $X_N$ based on empirical PCA on the training samples. In all experiments, and unless otherwise stated, we set \blue{$\alpha=1$}, \magenta{$\beta=1/\sqrt{2}$} and $L=100$, to ensure that the decoder is robust, i.e that it is not too sensitive to small pertubations in the input.

We will first compare with a method without compositions, using for the $g_i$ either sparse polynomials (Sparse) or low-rank polynomials using tree tensor networks (Low-Rank).

For both CPN-S and Sparse, we use hyperbolic cross sets for defining the background polynomial spaces $\mathcal{P}$.  We use  sparsity-inducing $\ell_1$-regularization and a LASSO-LARS homotopy algorithm that procudes a collection of candidate polynomial subspaces. For each subspace from this path, we re-estimate a classical least-squares approximation and we rely on cross-validation (leave-one-out) for model selection. For CPN-LR and Low-Rank, we use adaptive tree tensor networks using the library \texttt{tensap} \cite{tensap}, implementing  the algorithm from \cite{grelier:2018,michel2022}. We also rely on validation for model selection. 

For comparison, we will also  consider the quadratic manifold approximation method of  \cite{Barnett_2022, Geelen_2023} (Quadratic), with a decoder of the form 
$$D(a) = \bar u + \sum_{i=1}^n a_i \varphi_i + \sum_{1 \le i \le j \le n} a_i a_j \varphi_{i,j}$$
where $\bar u$ and the orthonormal vectors $\varphi_1, \dots, \varphi_n$ are chosen based on empirical PCA. 
The resulting approximation is in an affine space of dimension $N = n(n+3)/2.$ The vectors $ \varphi_{i,j}$, $1 \le i \le j \le n$, are solutions of the optimization problem 
$$
\min_{(\varphi_{i,j})} \sum_{k=1}^m \Vert u_k - D(E(u_k)) \Vert_X^2
$$
where $E(v) = ((v-\bar u,\varphi_i)_X)_{i=1}^n$. \red{We consider the basic version where the functions $\varphi_1, \dots,\varphi_n$ are taken as the first  components of the PCA, and also the improved algorithm from \cite{schwerdtner2024greedyconstructionquadraticmanifolds} (later called Greedy-Quadratic) where the functions 
$\varphi_1, \dots,\varphi_n$ are selected greedily among the  $N$ first components of the PCA. This allows for a fair comparison with our approach, which also selects greedily the set $I\subset \{1,\dots N\}$ that defines $X_n$ based on stability and approximation properties.}


We also compare to the approach   of \cite{geelen2024} using for the maps $g_i$ additive polynomials (sums of univariate  polynomials, i.e. $g_i(a) = \sum_{l=1}^n p_{i,l}(a_l)$ with $p_{i,l}$ a polynomial of degree $p$). We rely on an alternating minimizing  method (AM), which successively optimize over polynomials $(p_{i,l})$ and over the orthonormal functions $\varphi_i$, $1\le i\le N$, the latter optimization problem being an orthogonal Procrustes problem. The method is denoted (Additive-AM).

For both (Quadratic) and (Additive-AM) methods, we introduce an  $\ell_2$ regularization, as originally done by the authors, and use a grid search method to determine an optimal regularization term. 

\blue{Unless stated otherwise, all numerical experiments are performed in the mean-squared setting and we use as error measure the relative error}
\magenta{
$$
\text{ RE } = \frac{\sqrt{\sum \limits_{k=1}^m \Vert u_k - D(E(u_k)) \Vert_X^2}}{\sqrt{\sum \limits_{k=1}^m \Vert u_k \Vert_X^2}}
$$}
and denote by $ \text{RE}_{\text{train}} $ and $ \text{RE}_{\text{test}} $ the relative errors over the training and test sets, respectively. 

An implementation of the proposed method is available at  \href{https://github.com/JoelSOFFO/CPN_MOR.git}{https://github.com/JoelSOFFO/CPN\_MOR}.

\subsection{Korteweg-de Vries (KdV) equation}

We consider a classical benchmark in nonlinear dispersive waves which  describes how waves propagate through shallow water. More precisely, we consider the propagation 
of a soliton in a one-dimensional domain with periodic boundary conditions, described by the 
 Korteweg-de Vries (KdV) equation 

\begin{align*}
    \frac{\partial u}{\partial t} + 4 u \frac{\partial u}{\partial x} + \frac{\partial^3 u}{\partial x^3} = 0,
\end{align*}
where $ u : [-\pi, \pi] \times [0, 1] \to \mathbb{R}$ denotes the flow velocity of the wave. The initial condition given by $ u_0(x) = 1 + 24\,  \text{sech}^2(\sqrt 8 x) $. We consider the manifold  $M = \{u(t) : t\in [0,1]\}$. 
Functions are evaluated on a grid of $D= 256$ equispaced points over the spatial domain $ [-\pi, \pi] $ and identified with vectors in $X \in \Rbb^D$. The samples are collected by evaluating the solution $u(t)$ every $ \Delta t = 0.0002 $ time units over the time domain $ [0, 1] $, therefore resulting into $5001 $  samples (initial condition included). 
\red{We consider the same experimental setup as in \cite{geelen2024}, and investigate the method ability for  extrapolation.}
The first $m= 1001 $ samples, corresponding to times $ t \in [0, 0.2] $ are taken as training data, and we use as test samples the remaining $4000$ samples corresponding to times $t\in [0.2,1]$. 

We perform various experiments on KdV to illustrate the properties of the proposed method. 

\paragraph{Comparison with other methods}

\begin{table}[h]
\centering 

\begin{tabular}{|c|c|c|c|c|c|}
\hline
Method & $p$  & $n$ & $N$ & $ \text{RE}_{\text{train}} $ & $ \text{RE}_{\text{test}} $ \\ 
 \hline
 Linear & / &  2  & / & \magenta{$6.63 \times 10^{-1}$}  & \magenta{$ 6.91 \times 10^{-1} $ }\\ 
 \hline
 Quadratic &  2 & 2 & 5 & $ 5.33 \times 10^{-1}$  & \magenta{$ 5.66 \times 10^{-1}$} \\
 \hline
 \red{Greedy-Quadratic} & \red 2 & \red 2 & \red 5  & \red{$5.33 \times 10^{-1}$} & \red{$ 5.66 \times 10^{-1}$}
 \\
 \hline
Additive-AM & 5 & 2 & 43 & \magenta{$ 2.10 \times 10^{-1}$}  & \magenta{$ 5.47 \times 10^{-1}$}  \\
 \hline
 Sparse &  5 & 2 & 43 & \magenta{$ 1.73 \times 10^{-1}$}  & \magenta{$ 1.85 \times 10^{-1} $}  \\
 \hline
 Low-Rank & 5 & 2 & 43 & \magenta{$7.64 \times 10^{-2}$} & \magenta{$8.24 \times 10^{-2}$} \\
 \hline
 CPN-LR ($\epsilon = 10^{-4}$) &  5 & 2 & 43 & \magenta{$ 6.85 \times 10^{-5}$} & \magenta{$7.06 \times 10^{-5}$} \\ 
 \hline
\end{tabular}
\caption{(KdV) Comparison of methods for the same manifold dimension $ n = 2 $. For CPN, we use low-rank polynomials.}
\label{tab:kdv_comparison_table}
\end{table}

Figures \ref{fig:kdv_viz_solution} and \ref{fig:kdv_snapshots} illustrate the predicted solutions for the different methods. 
We observe that the solutions given by CPN predict very accurately the true solution over the whole time interval.
\red{It is important to mention that even though the Quadratic and Greedy-Quadratic approaches give similar results for $n=2$, we observe in practice that for this experiment and in \ref{Burgers-equation}, the Greedy-Quadratic outperforms the Quadratic method as the dimension $n$ increases.
}

\begin{figure}[h]
    \centering
    \subfigure[Exact solution]{\includegraphics[width=0.3\textwidth]{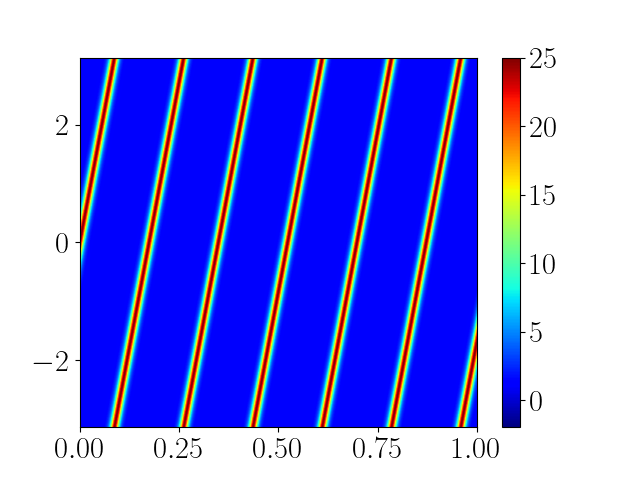}}
    \subfigure[Linear]{\includegraphics[width=0.3\textwidth]{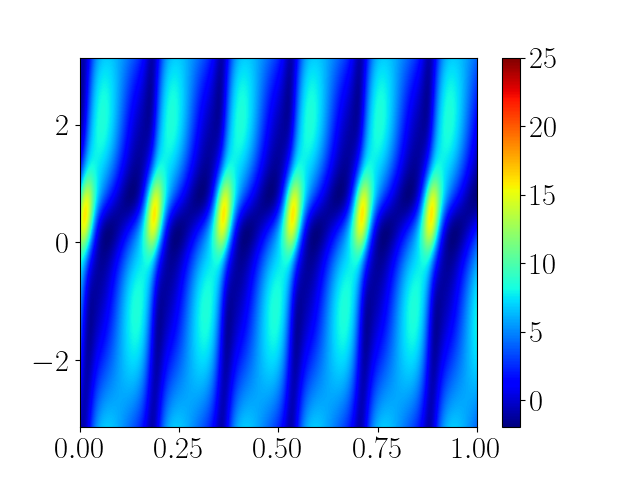}} 
    \subfigure[Quadratic]{\includegraphics[width=0.3\textwidth]{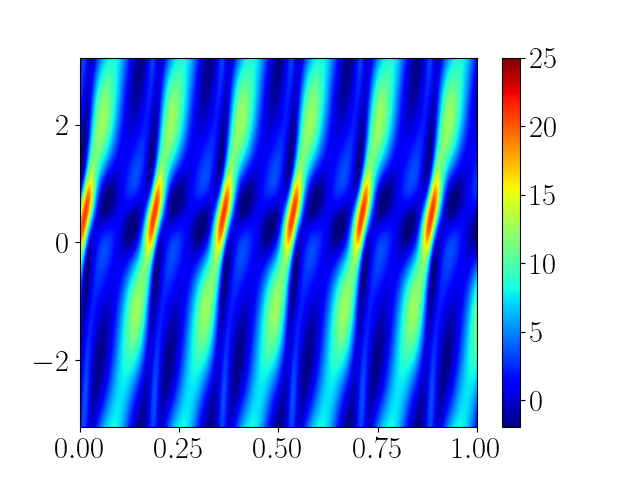}} 
    \subfigure[Additive-AM]{\includegraphics[width=0.3\textwidth]{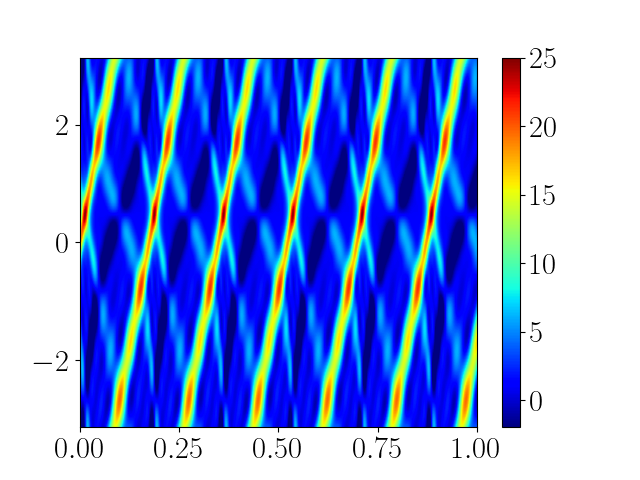}} 
    \subfigure[Sparse]{\includegraphics[width=0.3\textwidth]{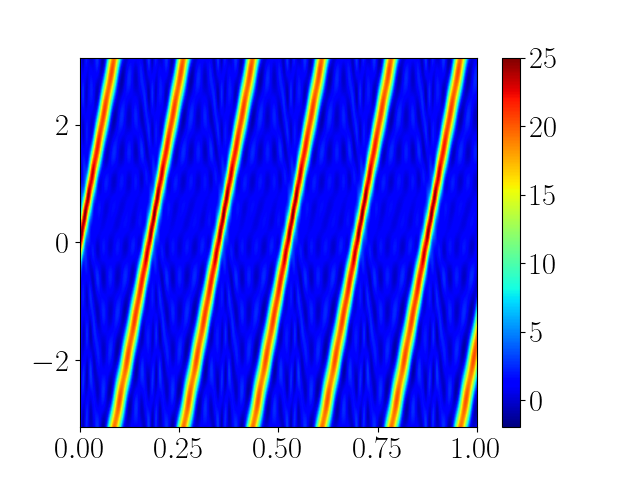}} 
    \subfigure[CPN-LR]{\includegraphics[width=0.3\textwidth]{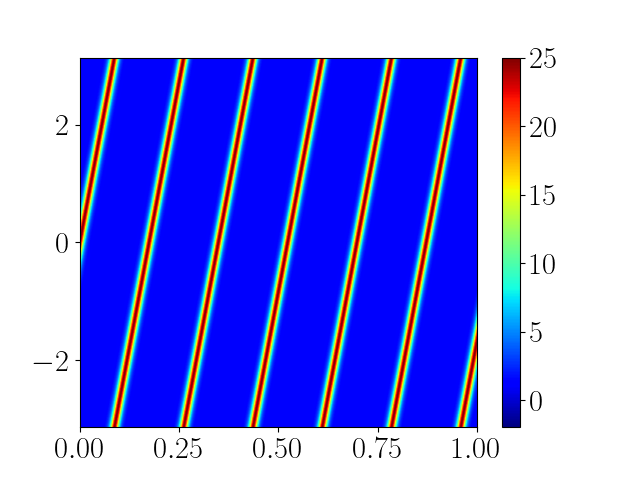}} 
    \caption{(KdV) Predictions for different methods, with $ n = 2 $.}
    \label{fig:kdv_viz_solution}
\end{figure}

\begin{figure}
    \centering
    \subfigure[$ u(\cdot, t) $ at $ t = 0.5 $]{\includegraphics[width=0.4\textwidth]{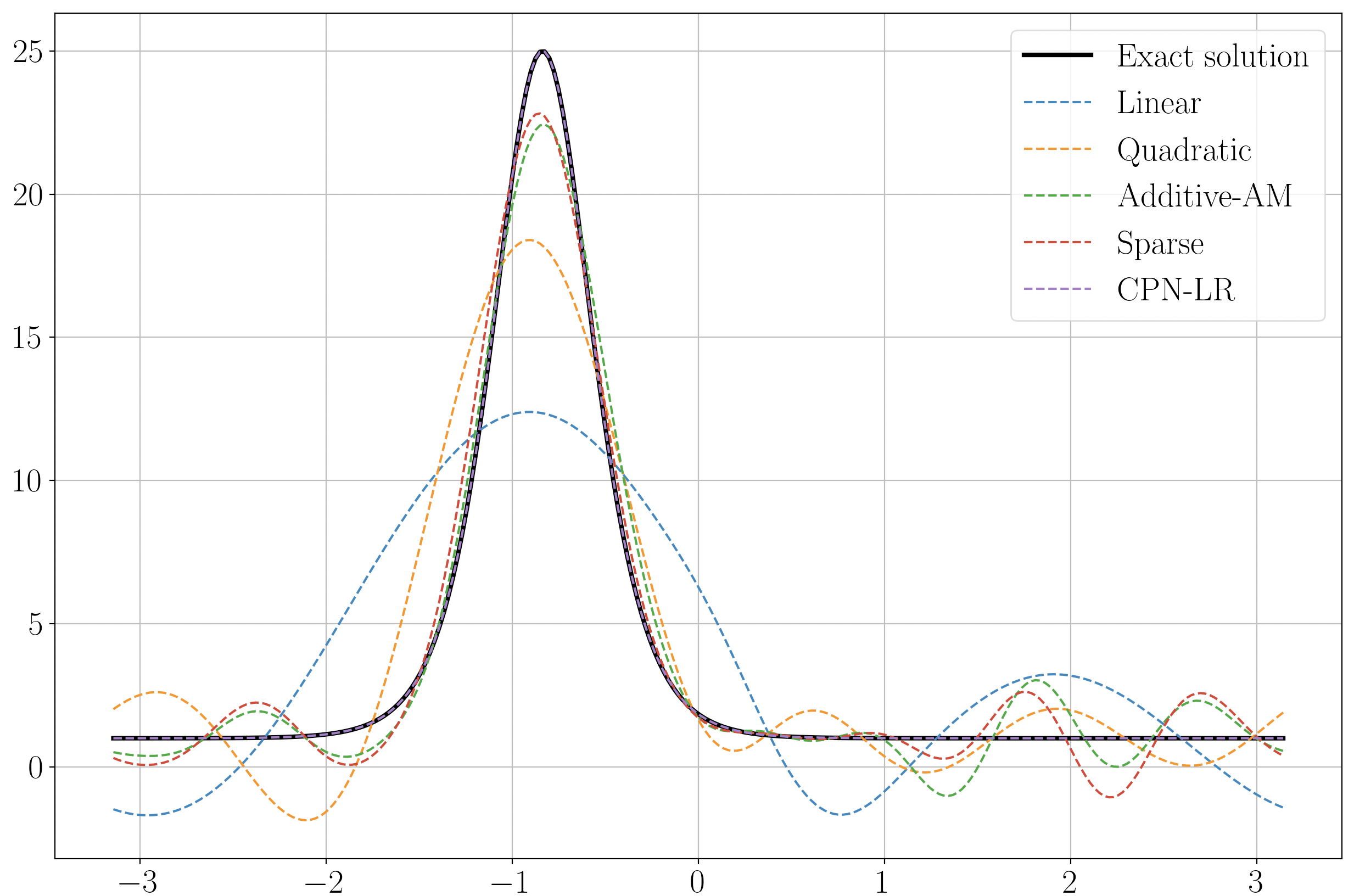}} 
    \subfigure[$ u(\cdot, t) $ at $ t = 1 $]{\includegraphics[width=0.4\textwidth]{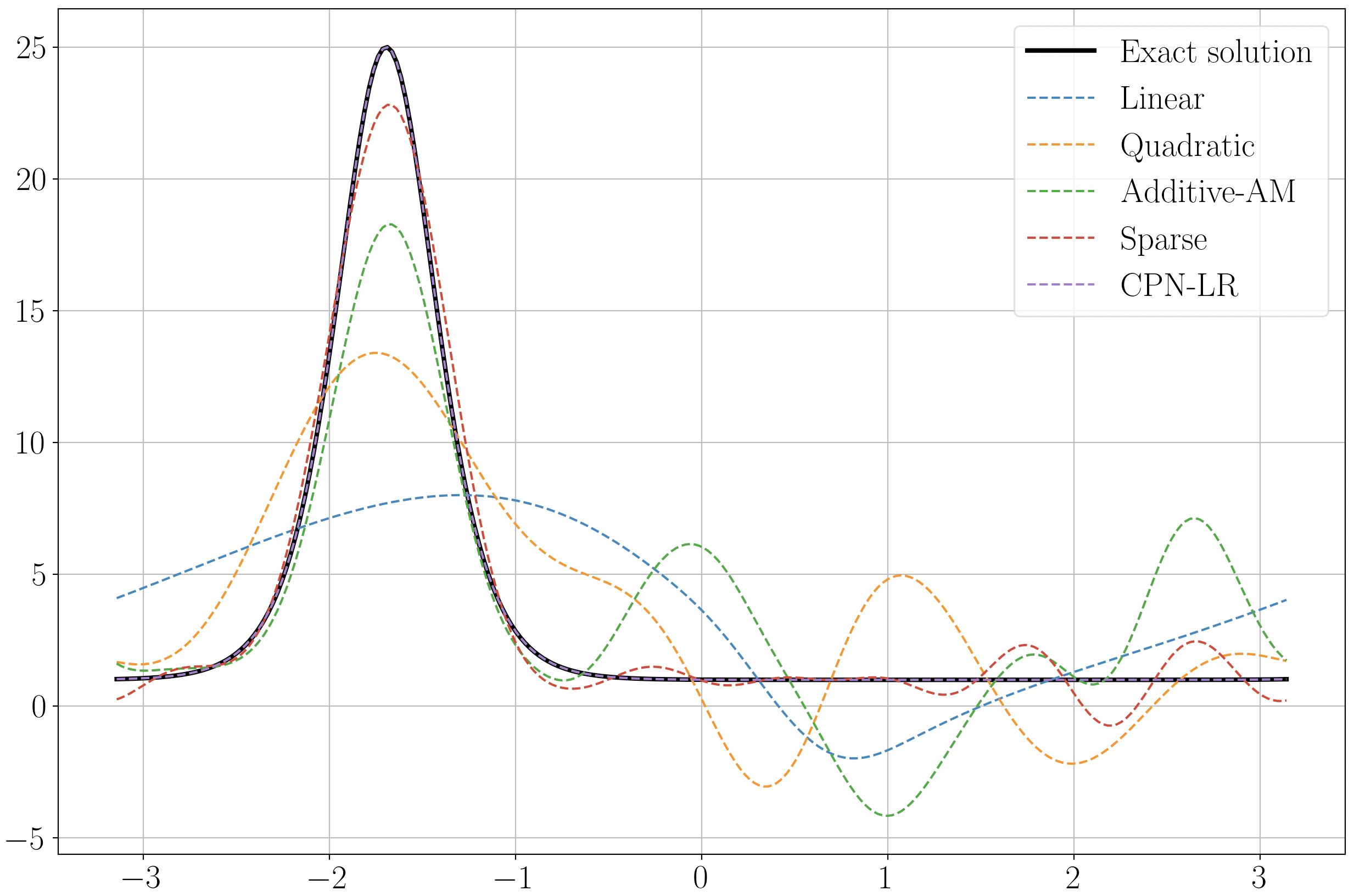}} 
    \caption{(KdV) Comparison between methods for two snapshots, with  $ n = 2 $.}
    \label{fig:kdv_snapshots}
\end{figure}

\paragraph{Influence of the polynomial degree}

In Table \ref{tab:different_p}, we illustrate the influence of the polynomial degree $p$ on the results of CPN. We observe that a high value of $p$ allows to capture higher nonlinearities in the relations between coefficients $a_i(u)$, hence a lower manifold dimension $n$.



\begin{table}[h]
\centering 

\begin{tabular}{|c|c|c|c|c|c|}
\hline
{Method} & $p$ & $n$ & $N$ & $ \text{RE}_{\text{train}} $ & $ \text{RE}_{\text{test}} $ \\ 
 \hline
\centering CPN-S & 3  & \magenta{8} & 43 & \magenta{$ 7.44 \times 10^{-5}$}  & \magenta{$ 7.68 \times 10^{-5}$} \\ \cline{2-6}

    & 4 & 7 & 43 & \magenta{$ 7.61 \times 10^{-5}$} & \magenta{$ 7.82 \times 10^{-5}$} \\ \cline{2-6}

    & 5 & 5 & 43 & \magenta{$ 7.30 \times 10^{-5} $} & \magenta{$ 7.54 \times 10^{-5} $} \\
 \hline
 \centering CPN-LR & 3 & 3 & 43 & \magenta{$ 6.77 \times 10^{-5}$} & \magenta{$7. \times 10^{-5}$} \\ \cline{2-6}

    & 4 & 2 & 43 & \magenta{$6.90 \times 10^{-5}$} & \magenta{$ 7.02 \times 10^{-5}$}  \\ \cline{2-6}

    & 5 & 2 & 43 & \magenta{$ 6.85 \times 10^{-5}$} & \magenta{$7.06 \times 10^{-5}$} \\
 \hline
\end{tabular}
\caption{(KdV) Results of CPN with different degrees $ p $ for $ \epsilon = 10^{-4} $.}
\label{tab:different_p}
\end{table}

\paragraph{Behavior of the algorithm}

In Table  \ref{tab:sparse_table}, we illustrate the results for various target precisions $\epsilon$. We observe that the algorithm returns an approximation satisfying the desired precision, with increasing dimensions $n$ and $N$ and number of compositions as $\epsilon$ decreases.  
The results also suggest that using low-rank approximation can lead to a smaller $ n $ than sparse polynomial approximation, due to the higher approximation power of the former. 
\blue{
Figure \ref{fig:pcnvssparse} shows the \textit{coefficient-wise errors} for Sparse and for CPN-S, that clearly illustrate the power of using compositions of polynomial maps. 
These errors are defined for $i\in I^c$ by
$$
e_i = \frac{\sqrt{\sum_{k=1}^m |a_i^k - g_i(a^k) |^2}}{ \sqrt{\sum_{k=1}^m \| u_k\|^2_X}},
$$
where the $g_i$ are either compositions of sparse polynomials maps (\textit{CPN-S}) or sparse polynomials of degree $p$ (\textit{Sparse}), and for the $k$-th sample, $a^k = (a_i^k)_{i \in I}$. 
}
 
\begin{table}[h]
\centering

\begin{tabular}{|c|c|c|c|c|c|c|}
\hline
{Precision} & {Method} & $n$ & $N$ & $N_{comp}$   & $ \text{RE}_{\text{train}} $  & $ \text{RE}_{\text{test}} $\\ \hline

\centering $\epsilon = 10^{-1}$ & CPN-S & 2 & 15 & \magenta 1 & \magenta{$6.26 \times 10^{-2}$} & \magenta{$6.48 \times 10^{-2}$} \\ \cline{2-7}
                                
                                &  CPN-LR & 2 & 15 & 1 & \magenta{$6.12 \times 10^{-2}$} & \magenta{$6.34 \times 10^{-2}$} \\ \hline
$\epsilon = 10^{-2}$ &  CPN-S & 3 & 25 & \magenta 2 & \magenta{$ 6.56 \times 10^{-3} $} & \magenta{$ 6.81 \times 10^{-3} $}\\ \cline{2-7}
                                
                        &  CPN-LR& 2 & 25 & 1 & \magenta{$ 6.02 \times 10^{-3}$} & \magenta{$ 6.23 \times 10^{-3}$} \\ \hline
$\epsilon = 10^{-3}$ &  CPN-S & 3 & 34 & \magenta 3 &  \magenta{$ 6.72 \times 10^{-4} $} & \magenta{$ 6.96 \times 10^{-4} $}\\ \cline{2-7}
                                
                        &  CPN-LR & 2 & 34 & \magenta 2 & \magenta{$ 6.17 \times 10^{-4}$} & \magenta{$ 6.38 \times 10^{-4}$} \\ \hline
$\epsilon = 10^{-4}$ &  CPN-S & 5 & 43 & \magenta 2 & \magenta{$ 7.30 \times 10^{-5} $} & \magenta{$ 7.54 \times 10^{-5}$}  \\ \cline{2-7}
                                
                        &  CPN-LR & 2 & 43 & \magenta 3 & \magenta{$ 6.85 \times 10^{-5}$} & \magenta{$7.06 \times 10^{-5}$} \\ \hline
$\epsilon = 10^{-5}$ &  CPN-S & 6 & 52 & \magenta 3 & \magenta{$ 6.70 \times 10^{-6}$} & $ 6.91 \times 10^{-6}$ \\ \cline{2-7}
                                
                        &  CPN-LR & 3 & 52 & \magenta 2 & 
 \magenta{$ 6.83 \times 10^{-6} $} & \magenta{$ 7.06 \times 10^{-6} $} \\ \hline
                                

\end{tabular}
\caption{(KdV) Comparison between CPN-S (sparse approximation) and CPN-LR (low-rank approximation) for $ p = 5 $ and various different values of target precision $\epsilon$ \magenta{(mean-squared setting)}. 
$N_{comp}$  indicates the maximum number of compositions.}
\label{tab:sparse_table}
\end{table}

\begin{figure}
    \centering
    \includegraphics[scale=0.3]{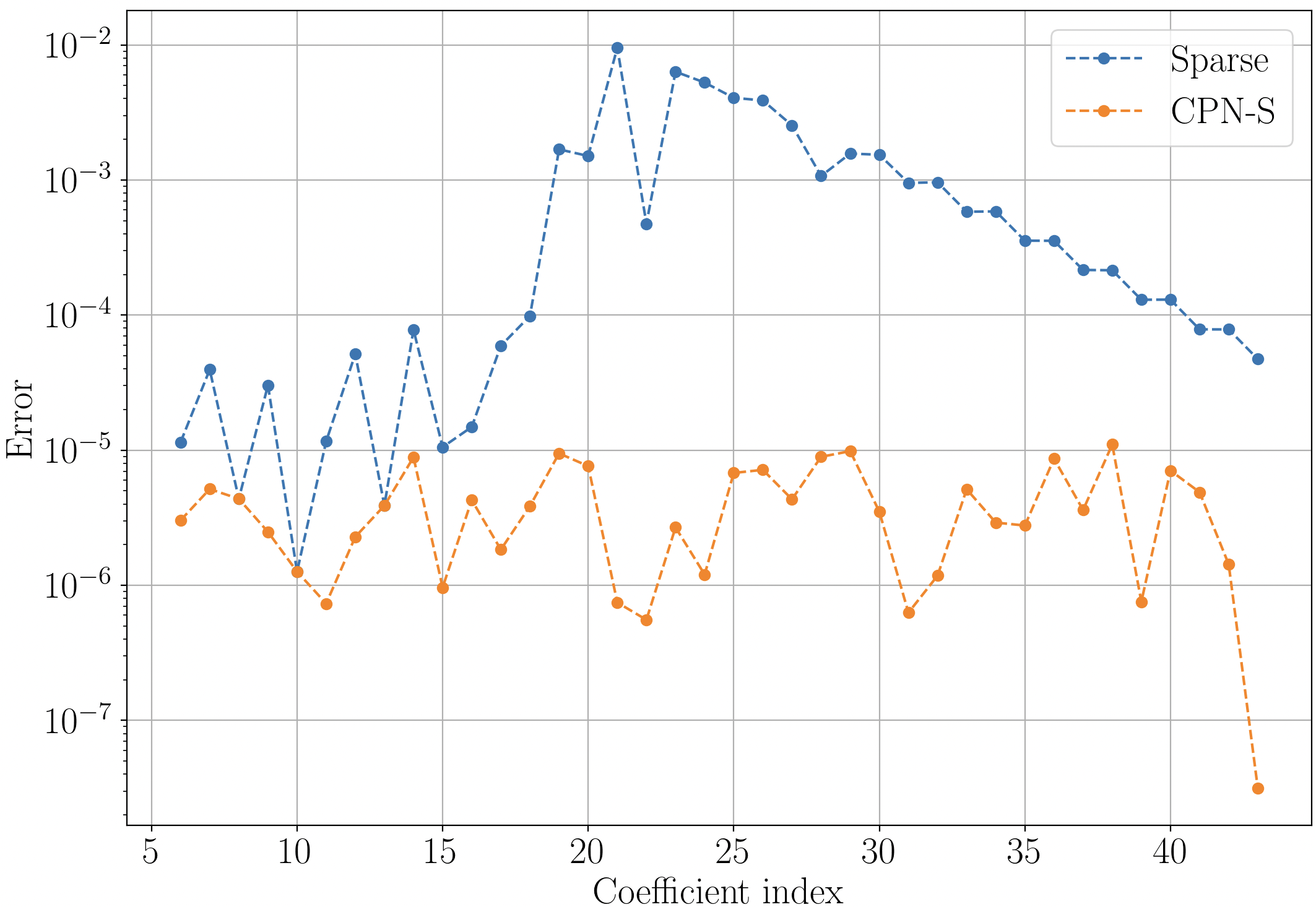}
    \caption{(KdV) Coefficient-wise errors for Sparse and CPN-S, with $p=5$ and $\epsilon=10^{-4}$.}
    \label{fig:pcnvssparse}
\end{figure}

 For CPN-S, the graphs of compositions for coefficients $a_{10}$, $a_{21}$ and $a_{41}$ can be visualized in Figure \ref{fig:kdv_graphs}. The coefficient $a_{10}$ is simply approximated in terms of the parameters $a$, whilst the coefficients $ a_{21} $ and $ a_{42} $ are expressed as compositions of polynomials. The full learning procedure for $ \epsilon = 10^{-4} $ and $ p = 5 $ is detailed in Table \ref{tab:kdv_learning_process}.

\begin{figure}
    \centering
    \subfigure[$a_{10}$]{\includegraphics[width=0.25\textwidth]{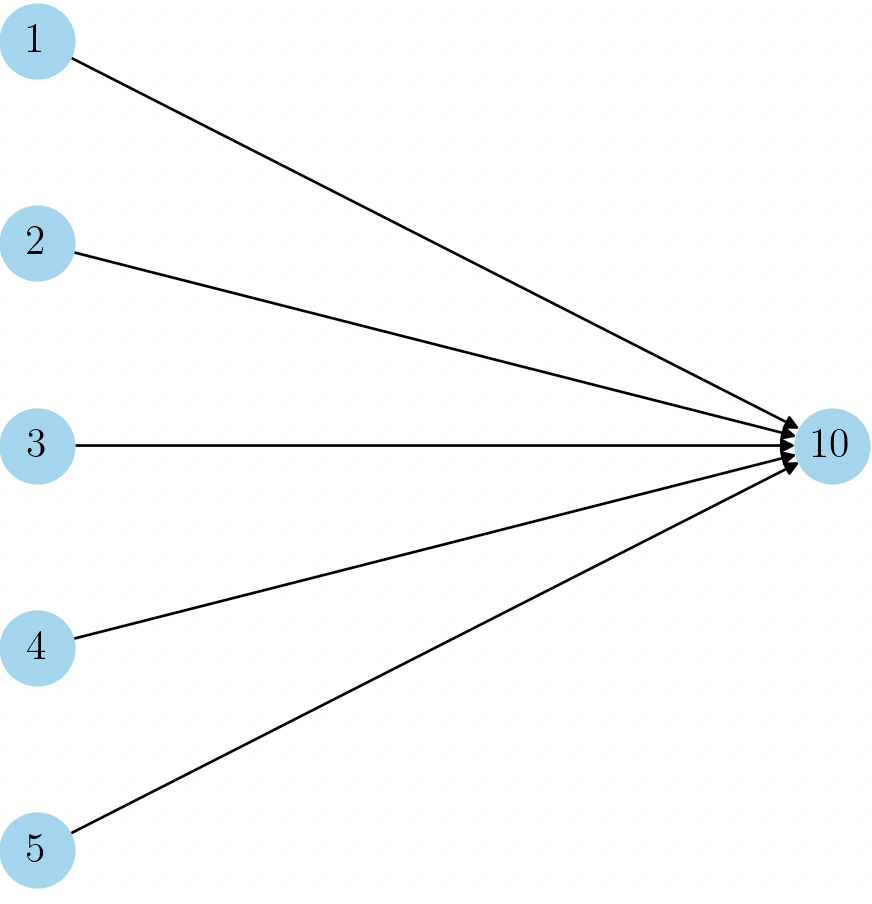}} 
    \subfigure[$ a_{21} $]{\includegraphics[width=0.3\textwidth]{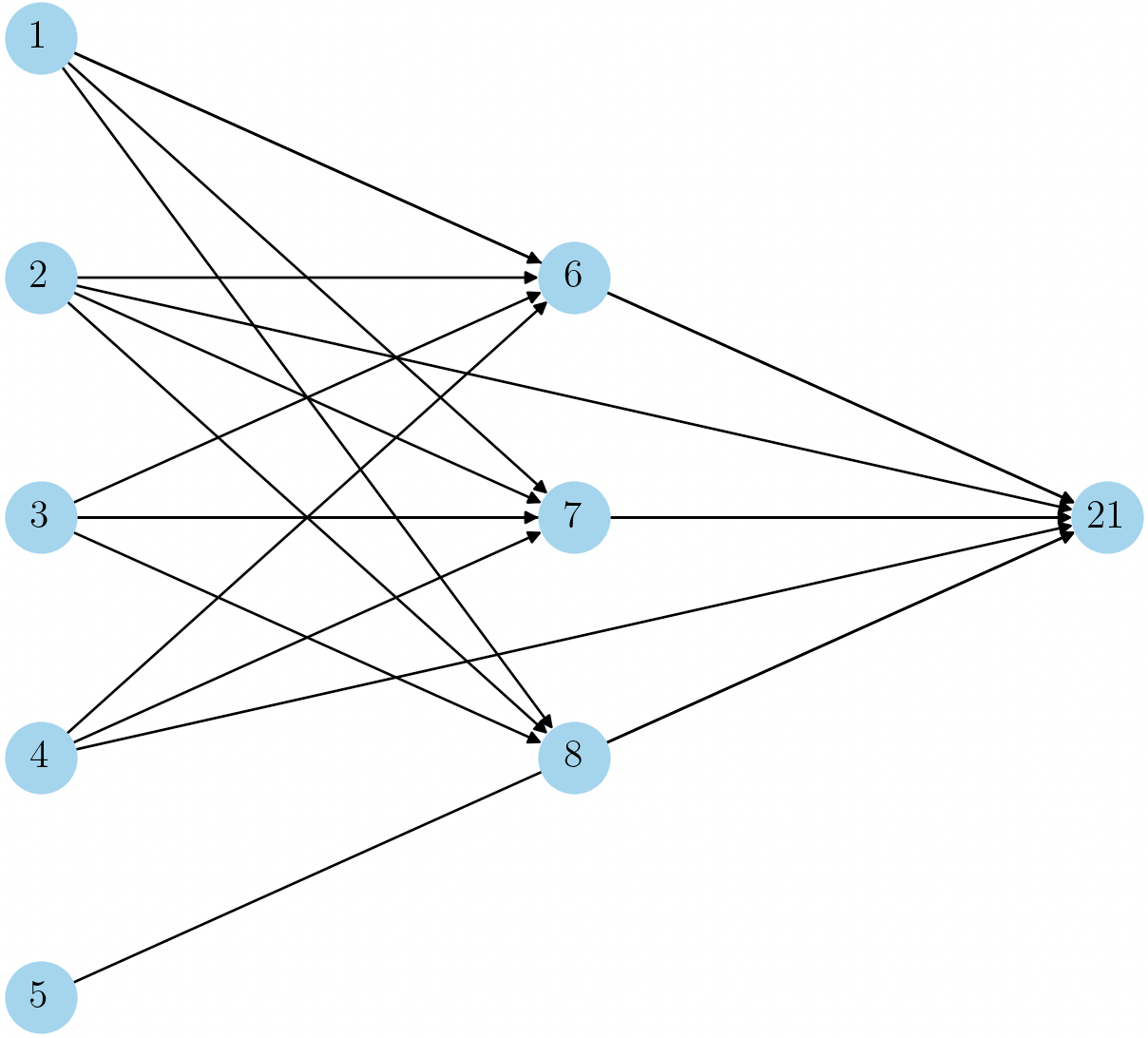}} 
    \subfigure[$ a_{42} $]{\includegraphics[width=0.3\textwidth]{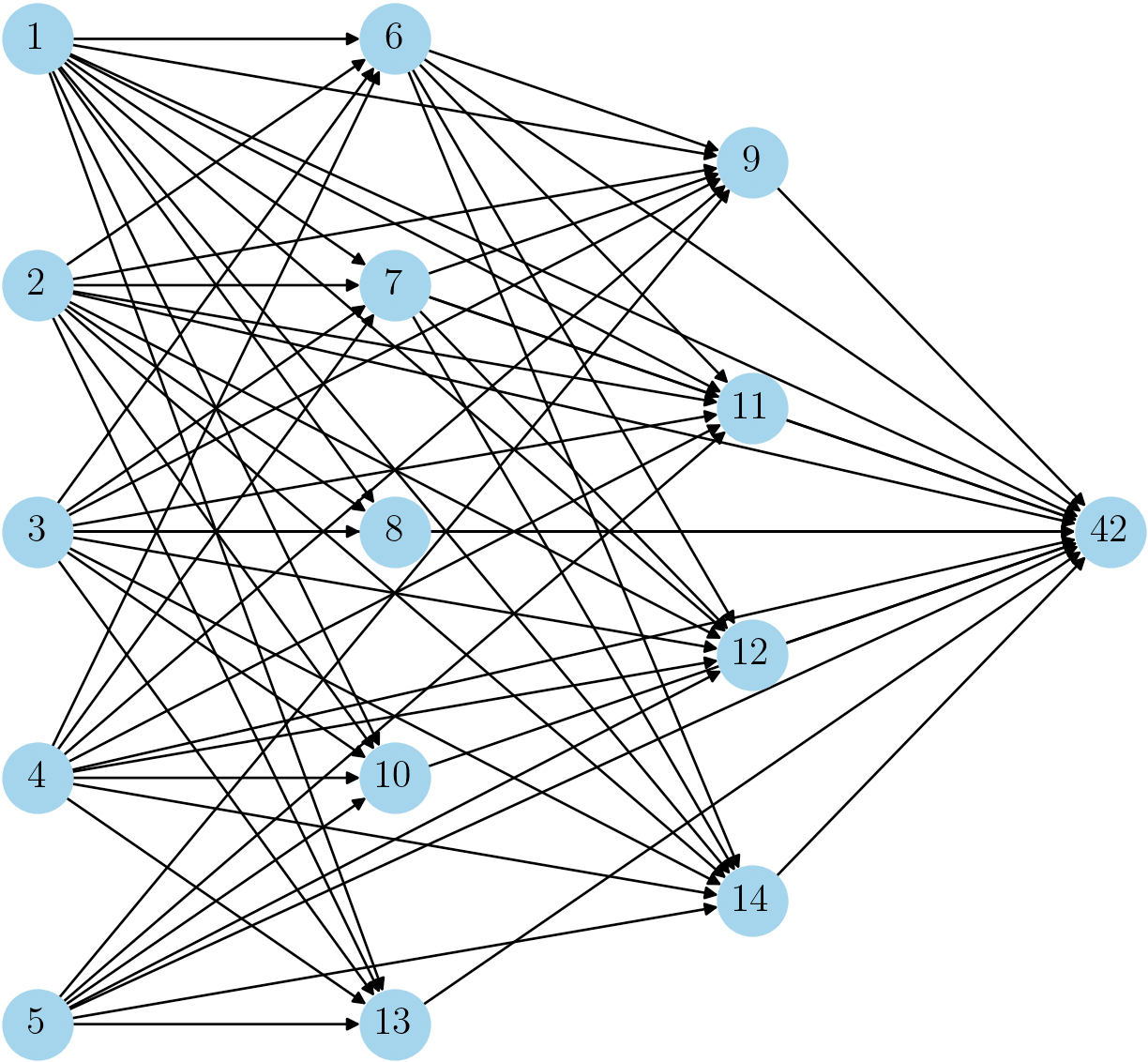}} 
    \caption{(KdV) Compositional networks for different coefficients, using CPN-S with $\epsilon=10^{-4}$ and $p=5$.}
    \label{fig:kdv_graphs}
\end{figure}

\begin{table}[H]
\centering 

\begin{tabular}{|c|c|c|}
\hline
\textbf{Step} & \textbf{Indices of input coeffs.} & \textbf{Indices of learnt coeffs.} \\
\hline
1 & 1 & / \\
\hline
2 & 1, 2 & / \\
\hline
3 & 1, 2, 3 & / \\
\hline
4 & 1, 2, 3, 4 & \magenta 6 \\
\hline
5 & 1, 2, 3, 4, 5 & \magenta{8, 10, 13} \\
\hline
6 & 1, 2, 3, 4, 5, \textbf 6 & \magenta{7, 9, 11, 12, 14, 15, 17} \\
\hline
7 & 1, \ldots, 5, \textbf 6, \textbf 7 & \magenta{16, 18, 19, 20} \\
\hline
8 & 1, \ldots, 5, \textbf{6, 7, 8} & \magenta{21, 22, 24, 25, 27, 29} \\
\hline
9 & 1, \ldots, 5, \textbf{6, 7, 8, 9} & \magenta{23, 26, 28} \\
\hline
10 & 1, \ldots, 5, \textbf{6, \ldots, 10} & \magenta{30} \\
\hline
11 & 1, \ldots, 5, \textbf{6, \ldots, 11} & \magenta{32, 33, 34, 36, 38} \\
\hline
12 & 1, \ldots, 5, \textbf{6, \ldots, 12} & \magenta{31, 35, 37} \\
\hline
13 & 1, \ldots, 5, \textbf{6, \ldots, 13} & \magenta{39, 40, 41} \\
\hline
14 & 1, \ldots, 5, \textbf{6, \ldots, 14} & \magenta{42} \\
\hline
15 & 1, \ldots, 5, \textbf{6, \ldots, 15} & \magenta{43} \\
\hline
\end{tabular}
\caption{(KdV) Learning procedure of CPN-S for $ p = 5 $ and $ \epsilon = 10^{-4}$. To reach the target precision,  $N=43$ is required. Coefficients are progressively learnt throughout the different steps of the algorithm. At step $ j $, $ a_j $ is added as input variable. If already learnt at a previous step, its approximation is used instead (indices in bold), leading to the compositional structure. The final dimension $ n = 5 $ corresponds to the number of coefficients that were not learnt during the process.}
\label{tab:kdv_learning_process}
\end{table}

\paragraph{Stability of the decoder}

In Table \ref{tab:different_L}, we illustrate the influence of the prescribed upper bound $L$ for the Lipschitz constant, with prescribed precision $\epsilon=10^{-4}$ and polynomial degree $p=5$. The algorithm was able to construct an approximation satisfying the target precision and stability conditions.   
As expected, we observe that imposing a smaller Lipschitz constant results in a higher manifold dimension $n$. 

\begin{table}[H]
\centering 

\begin{tabular}{|c|c|c|c|c|c|}
\hline
{Method} & $L$ & $n$ & $N$ & $ \text{RE}_{\text{train}} $ & $ \text{RE}_{\text{test}} $ \\ 
 \hline
\centering CPN-S & 2  & 14 & 43 & \magenta{$ 7.21 \times 10^{-5}$}  & $ 7.46 \times 10^{-5}$ \\ \cline{2-6}

    & 10 & 6 & 43 & \magenta{$ 7.25 \times 10^{-5}$} & \magenta{$ 7.49 \times 10^{-5}$} \\ \cline{2-6}

    & 100 & 5 & 43 & \magenta{$ 7.30 \times 10^{-5} $} & \magenta{$ 7.54 \times 10^{-5}$} \\
 \hline
 \centering CPN-LR & 2 & 12 & 43 & \magenta{$7.13 \times 10^{-5}$} & \magenta{$7.38 \times 10^{-5}$} \\ \cline{2-6}

    & 10 & 7 & 43 & \magenta{$6.99 \times 10^{-5}$} & \magenta{$7.10 \times 10^{-5}$} \\ \cline{2-6}

    & 100 & 2 & 43 & \magenta{$ 6.85 \times 10^{-5}$} & \magenta{$7.06 \times 10^{-5}$} \\
 \hline
\end{tabular}
\caption{(KdV) Results of CPN for different Lipschitz constants $L$, with $ \epsilon = 10^{-4} $ and $p=5$.}
\label{tab:different_L}
\end{table}

\paragraph{Worst-case setting}

\blue{
We now report some results to illustrate the performance of the method for the worst-case setting.
We define the relative worst-case error by
$$
RE^{\infty} \coloneqq \sup \limits_{u \in \{ u_1, \ldots, u_k \}} \| u - D(E(u)) \|_X / \sup \limits_{u \in \{ u_1, \ldots, u_k \}} \|u\|_X.$$
We choose here $\beta=1/2$.
Table \ref{tab:sparse_table_worst_case} contains the results obtained for different levels of precision.
As the error control in the worst-case is more challenging than the mean-squared case, the dimensions $n$ observed in the former are   higher, as expected. However, we observe that the algorithm still returns an encoder-decoder pair with the desired worst-case accuracy. A weaker version considering quantiles of errors (not worst case errors) could reduced the required manifold dimension.  
}
\begin{table}[h]
\centering

\begin{tabular}{|c|c|c|c|c|c|c|}
\hline
{Precision} & {Method} & $n$ & $N$ & $N_{comp}$   & $ \text{RE}_{\text{train}}^{\infty} $  & $ \text{RE}_{\text{test}}^{\infty} $\\ \hline

\centering $\epsilon = 10^{-1}$ & CPN-S & 7 & 23 & 1 & $4.74 \times 10^{-2}$ & $4.75 \times 10^{-2}$ \\ \hline
$\epsilon = 10^{-2}$ &  CPN-S & 12 & 34 & 1 & $2.73 \times 10^{-3}$ & $2.73 \times 10^{-3}$\\ \hline
$\epsilon = 10^{-3}$ &  CPN-S & 16 & 42 & 1 & $4.31 \times 10^{-4}$ & $4.31 \times 10^{-4}$ \\ \hline
$\epsilon = 10^{-4}$ &  CPN-S & 24 & 52 & 2 & $4.98 \times 10^{-5}$ & $4.99 \times 10^{-5}$ \\ \hline

\end{tabular}
\caption{(KdV) CPN-S (sparse approximation) in the worst-case setting for $ p = 5 $ and various different values of target precision $\epsilon$.}
\label{tab:sparse_table_worst_case}
\end{table}

\subsection{Allen-Cahn equation}

The Allen-Cahn equation describes the process of phase separation in multi-component alloy systems. We consider a space-time domain $\Omega \times [0, T] $ with 
$\Omega=(-1,1)$, $ T = 60 $. The equation on the phase field variable $ u : \Omega \times [0, T] \rightarrow \mathbb R $ is 
$$
    \frac{\partial u}{\partial t} = \eta^2 \Delta u + u - u^3,\quad \text{on} \quad \Omega \times [0, T],
$$
with $\eta =10^{-2} $, and with boundary conditions 
$
u(-1, t) = -1$ and $u(1, t) = 1
$ for $t\in [0,T]$, and initial condition 
$
u(x, 0) = \lambda x + (1-\lambda) \text{sin}(-1.5 \pi x)$ for $x\in \Omega$, where $\lambda$ is a uniform random variable with distribution   $U([0.5, 0.6])$ \cite{geelen2024}.

The spatial domain is discretized into $D=512$ equispaced points. A time step size $\Delta t=0.1 $ is fixed to discretize the time interval $ [0, 60] $. For the training set, three values of $ \lambda $ are considered, $ \lambda \in \{0.5, 0.55, 0.60\} $. For the test set, $ 10 $ values of $ \lambda $ are uniformly sampled in $ [0.5, 0.6] $. We therefore have $1803$ training data and $6010$ test data, including the initial conditions. 

We run our method CPN-LR with a target precision $\epsilon =10^{-3}$ and  a polynomial degree $p=3$, which results in a dimension $N = 7$ and a manifold dimension $ n= 2$ (selected by the algorithm).  We  compare  different  methods  in  Table  \ref{tab:allen_cahn_comparisons}  for  the same manifold dimension $n= 2$.  We again observe that CPN-LR outperforms SOTA methods by one order of magnitude. 

\begin{table}[h!]
\centering 

\begin{tabular}{|c|c|c|c|c|c|}
\hline
 Method & $p$  & $n$ & $N$ & $ \text{RE}_{\text{train}} $ & $ \text{RE}_{\text{test}} $ \\ 
 \hline
 Linear & / &  2  & / & $3.38 \times 10^{-2}$ & $3.35 \times 10^{-2}$ \\ 
 \hline
 Quadratic & 2 & 2 & 5 & $ 1.87 \times 10^{-2}$ & $ 1.84 \times 10^{-2}$ \\
 \hline
 \red{Greedy-quadratic} & \red 2 & \red 2 & \red 5 & \red{$ 1.87 \times 10^{-2}$} & \red{$ 1.84 \times 10^{-2}$} \\
 \hline
 Additive-AM & 3 & 2 & 7 & $ 3.42 \times 10^{-3}$ & $3.45 \times 10^{-3}$ \\
  \hline
 Sparse & 3 & 2 & 7 & $ 2.34 \times 10^{-3}$  & $ 2.19 \times 10^{-3} $ \\
 \hline
 Low-rank & 3 & 2 & 7 & $1.05 \times 10^{-3} $ & $ 1. \times 10^{-3}$ \\
 \hline
 CPN-LR $(\epsilon=10^{-3})$ & 3 & 2 & 7 & \magenta{$ 3.76 \times 10^{-4}$} & \magenta{$ 3.56 \times 10^{-4}$} \\ 
 \hline
\end{tabular}
\caption{(Allen-Cahn) Comparison of methods for the same manifold dimension $n=2$.}
\label{tab:allen_cahn_comparisons}
\end{table}

Figure \ref{fig:allen_cahn_graphs} illustrates the compositional networks for three coefficients. 

\begin{figure}[h]
    \centering
    \subfigure[$a_{3}$]{\includegraphics[width=0.3\textwidth]{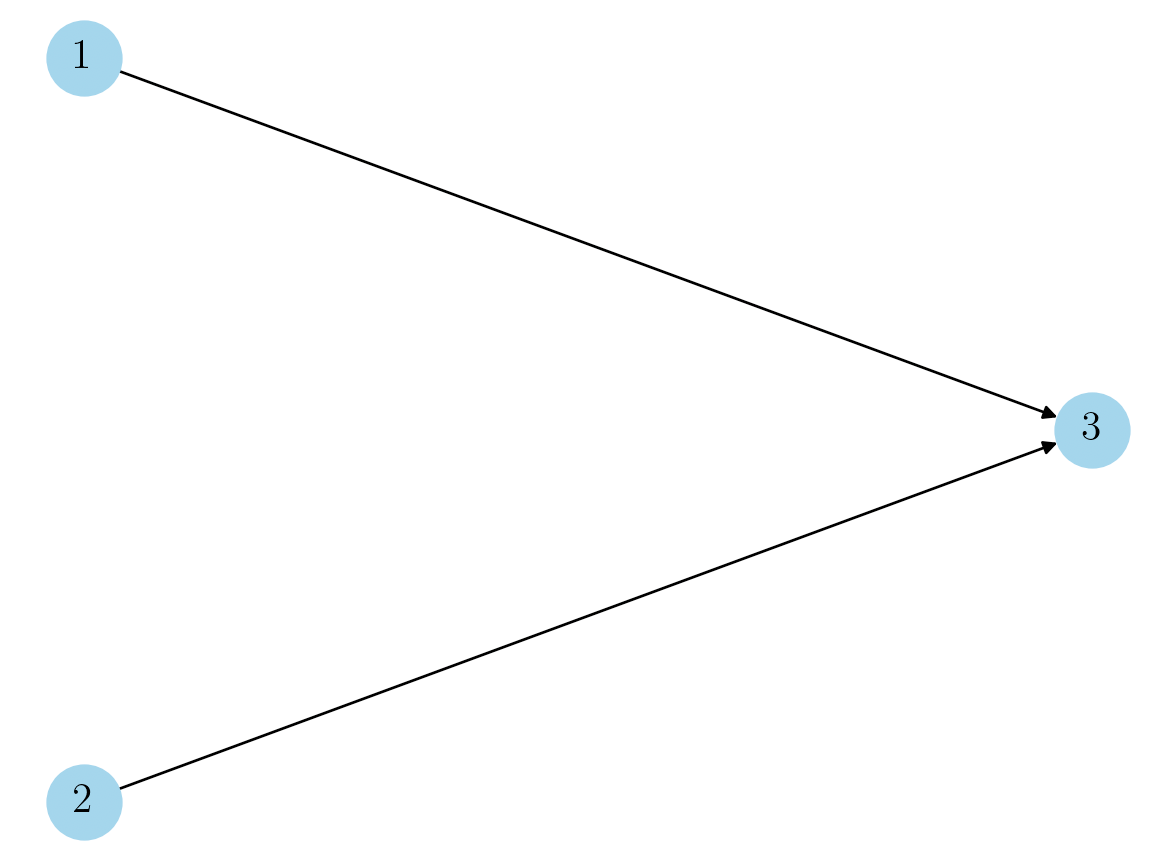}} 
    \subfigure[$ a_{5} $]{\includegraphics[width=0.3\textwidth]{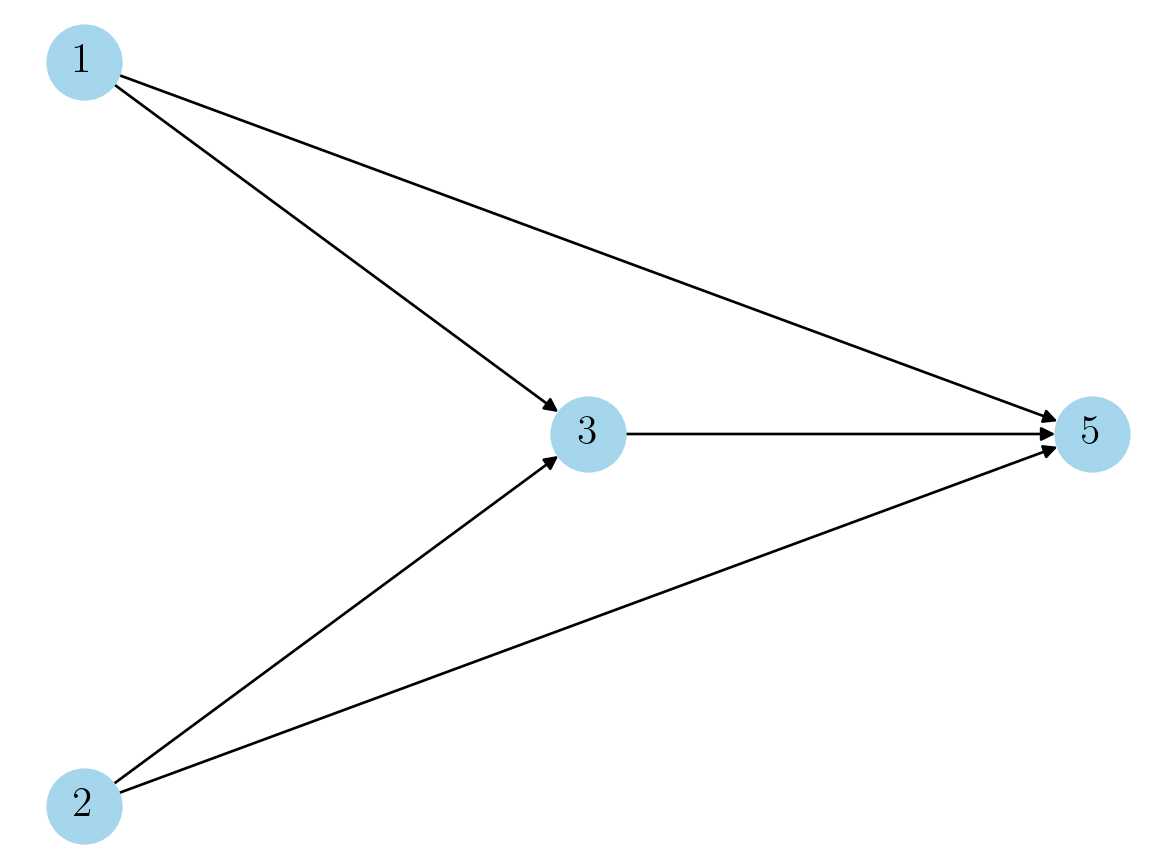}} 
    \subfigure[$ a_{6} $]{\includegraphics[width=0.3\textwidth]{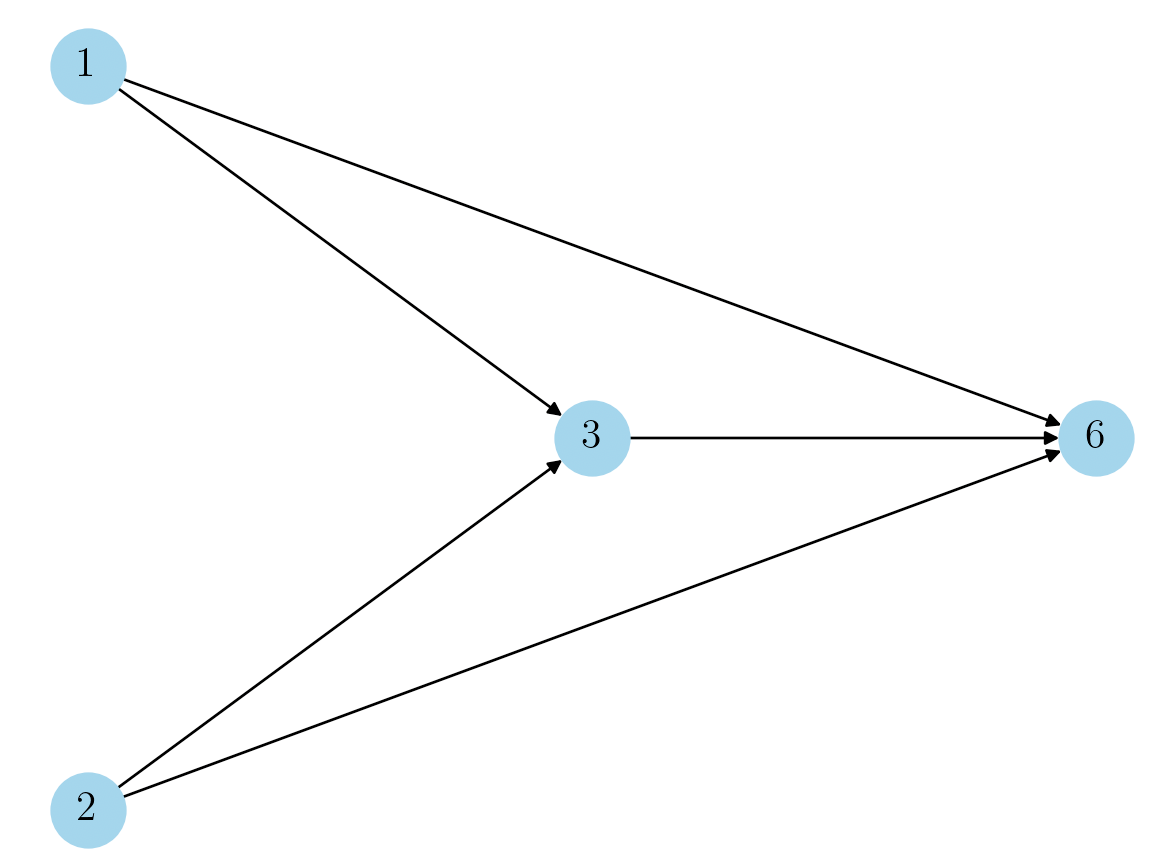}} 
    \caption{(Allen-Cahn) Networks for different coefficients, using CPN-LR with $\epsilon=10^{-3}$ and $p=3$.}
    \label{fig:allen_cahn_graphs}
\end{figure}

Figure \ref{fig:allen_cahn_viz_solution} shows the predicted solutions for $\lambda = 0.55188$. It illustrates the capacity of CPN-LR to provide a very accurate approximation  with a manifold of dimension $2$. 

 \begin{figure}[H]
    \centering
    \subfigure[Exact solution]{\includegraphics[width=0.4\textwidth]{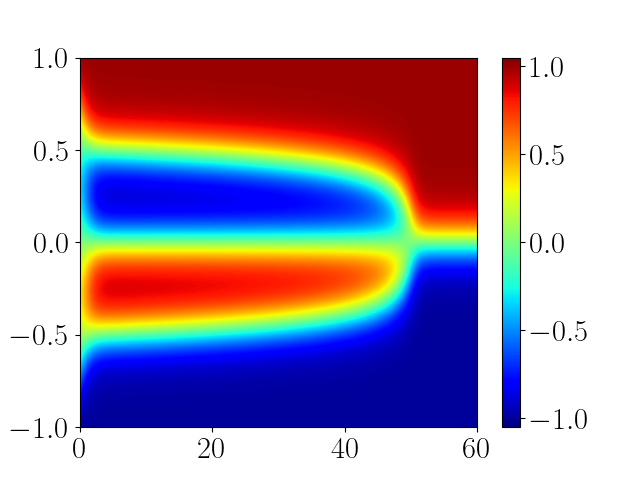}}
    \subfigure[CPN-LR]{\includegraphics[width=0.4\textwidth]{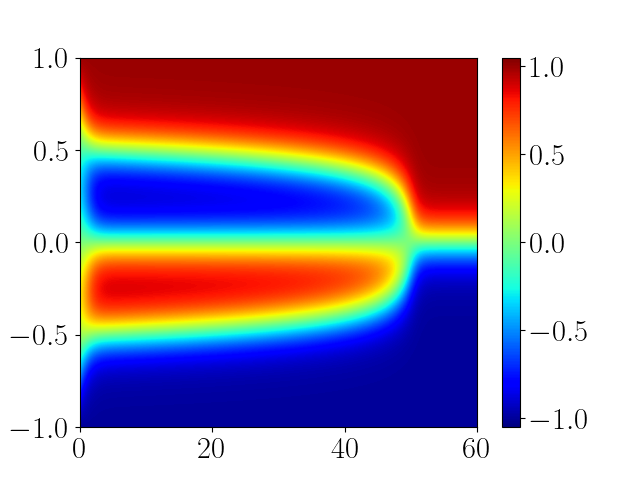}} 
    \caption{(Allen-Cahn) Predictions of CPN-LR with $n=2$, $\lambda = 0.55188$.}
    \label{fig:allen_cahn_viz_solution}
\end{figure}



\subsection{Inviscid Burgers equation}
\label{Burgers-equation}
We now assess the performance of the method on the solution manifold of the proposed 2D inviscid Burgers equation in \cite{Barnett2023Nov}, set over the space time domain $\Omega \times [0,T]$, with $\Omega = (0, 50)^2$, $T = 10$, and parametrized by a scalar parameter $\lambda \in \mathbb R$ involed in the boundary conditions.
Given $\lambda \in \mathbb R$, the equation governing the velocity field $(u_\lambda, v_\lambda) : \Omega \times [0,T] \rightarrow \Rbb^2$ reads
\begin{equation*}
    \left\{ 
    \begin{array}{l}
        \displaystyle
        \frac{\partial u_\lambda}{\partial t} + \frac{1}{2} \Big( \frac{\partial u_\lambda^2}{\partial x} + \frac{\partial u_\lambda v_\lambda}{\partial y}  \Big) = 0, 
        \qquad 
        \displaystyle 
        \frac{\partial v_\lambda}{\partial t} + \frac{1}{2} \Big(\frac{\partial u_\lambda v_\lambda}{\partial x} + \frac{\partial v_\lambda^2}{\partial y} \Big) = 0,
        \\[15pt]
        \begin{array}{ll}
            \text{initial conditions} &
            \\
            \quad u_\lambda(x, y, 0) = v_\lambda(x, y, 0) = 1, & x, y \in (0, 50),
            \\[8pt]
            \text{boundary conditions} &
            \\ 
            \\[5pt]
            \quad u_\lambda(x, 0, t) = 0, \quad v_\lambda(x, 0, t) = 0, & \quad x \in (0, 50), \enspace t \in (0, T).
        \end{array}
    \end{array}
    \right.
\end{equation*}
We consider here the solution manifold $ M = \big\{  (u_\lambda, v_\lambda): \lambda \in [1.5, 2.5] \big\} $, where $(u_\lambda, v_\lambda)$ are high-fidelity numerical solutions discretized on a regular grid of $ D = 250 \times 250 $ nodes in $\overline \Omega$ and sampled every $\Delta t = 0.03$ time units over the time domain $[0, 10]$.
The training samples are computed for $\lambda \in \{1.5, 1.7, 1.9, 2.2, 2.5\}$ and $t \in \{ k \Delta t \: : \: 0 \leq k \leq 300 \}$, resulting into $1501$ training data (the initial condition appearing without repetition).
For the test set, we uniformly sample $8$ additional values of $\lambda$ in the interval $[1.5, 2.5] $, yielding $2401$ test data.
 
We run our method CPN-S with a target precision $\epsilon =5.10^{-3}$ and  a polynomial degree $p=8$, which results in a dimension $N = 89$ and a manifold dimension $ n= 9$ (selected by the algorithm).
We  compare  different  methods  in  Table  \ref{tab:burgers_comparison_table}  for  the same manifold dimension $n= 9$.  
We again observe that CPN-S still performs better than SOTA methods.
\magenta{For this example, the Additive-AM method was computationally intractable. Note that here, Greedy-Quadratic is slightly better than the Quadratic method.}

\begin{table}[h]
\centering 

\begin{tabular}{|c|c|c|c|c|c|}
\hline
 Method & $p$  & $n$ & $N$ & $ \text{RE}_{\text{train}} $ & $ \text{RE}_{\text{test}} $ \\ 
 \hline
 Linear & / &  9  & / & $5.60 \times 10^{-2}$  & $5.01 \times 10^{-2}$ \\ 
 \hline
 Quadratic &  2 & 9 & 54 & $ 2.46 \times 10^{-2}$  & $ 2.28 \times 10^{-2}$  \\ 
 \hline
  \red{Greedy-quadratic} &  \red 2 & \red 9 & \red{54} & \red{$ 1.31 \times 10^{-2}$}  & \red{$ 1.71 \times 10^{-2}$}  \\ 
 \hline
Additive-AM & 8 & 9 & 89 & /  & /  \\
 \hline
  Low-Rank & 8 & 9 & 89 & $2.46\times 10^{-2}$ & $2.33 \times 10^{-2}$ \\
 \hline
 Sparse &  8 & 9 & 89 & $ 1.62 \times 10^{-2} $  & $ 1.62 \times 10^{-2} $  \\
 \hline
 CPN-S ($\epsilon = 5.10^{-3}$) &  8 & 9 & 89 & $ 4.74 \times 10^{-3} $ & $ 4.93 \times 10^{-3} $ \\ 
 \hline
\end{tabular}
\caption{(2D Burgers) Comparison of methods for the same manifold dimension $ n = 9 $.}
\label{tab:burgers_comparison_table}
\end{table}

Figure \ref{fig:burgers_viz_solution} shows the  solutions predicted by CPN-S for a particular $\lambda$ and two time values.

\begin{figure}[h!]
    \centering
    \subfigure[Exact solution, $t=5$]{\includegraphics[width=0.4\textwidth]{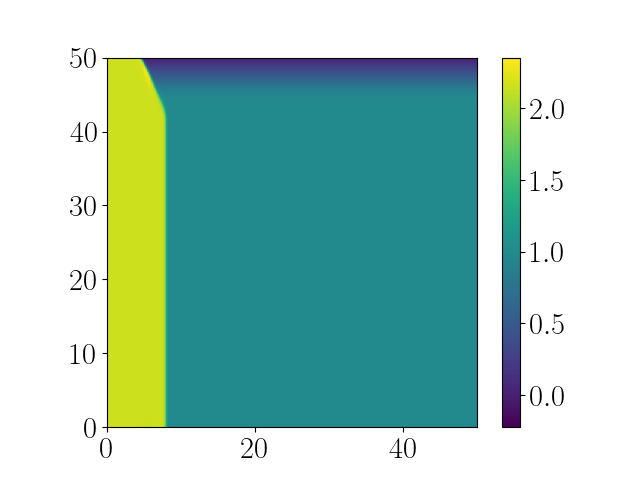}} 
    \subfigure[CPN-S, $t=5$]{\includegraphics[width=0.4\textwidth]{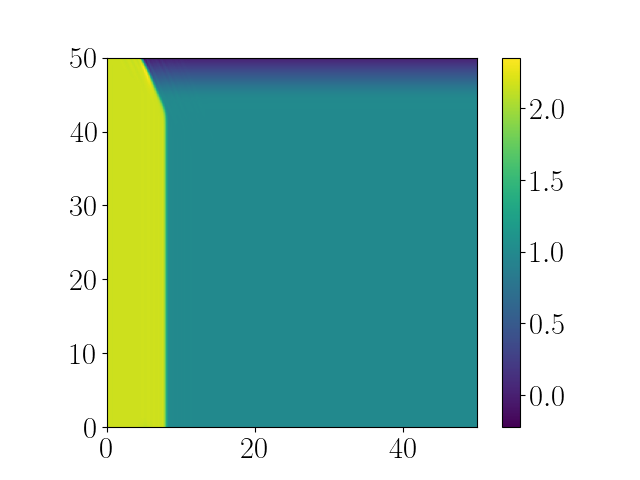}} 
    \subfigure[Exact solution $t=9.7$]{\includegraphics[width=0.4\textwidth]{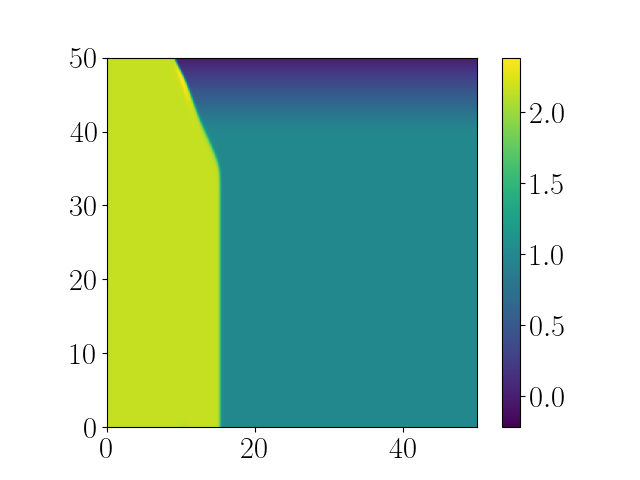}} 
    \subfigure[CPN-S, $t=9.7$]{\includegraphics[width=0.4\textwidth]{New_figures/burgers_exact_t_9.png}}
    \caption{(2D Burgers) Predictions of CPN-S with $ n = 9 $, for $\lambda = 2.15717$, at $t=5$ (top) and $t=9.7$ (bottom).}
    \label{fig:burgers_viz_solution}
\end{figure}

Figure \ref{fig:burgers_wise_viz_solution} shows errors for different methods for a particular value of $\lambda$.

\begin{figure}[h!]
    \centering
    \subfigure[Linear]{\includegraphics[width=0.3\textwidth]{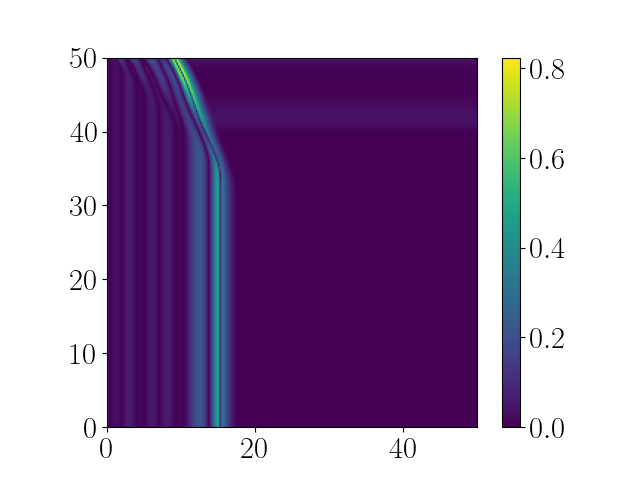}} 
    \subfigure[Quadratic]{\includegraphics[width=0.3\textwidth]{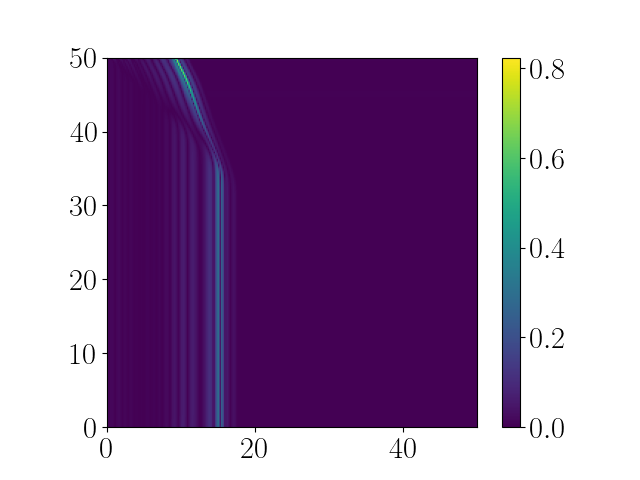}} 
    \subfigure[Greedy-quadratic]{\includegraphics[width=0.3\textwidth]{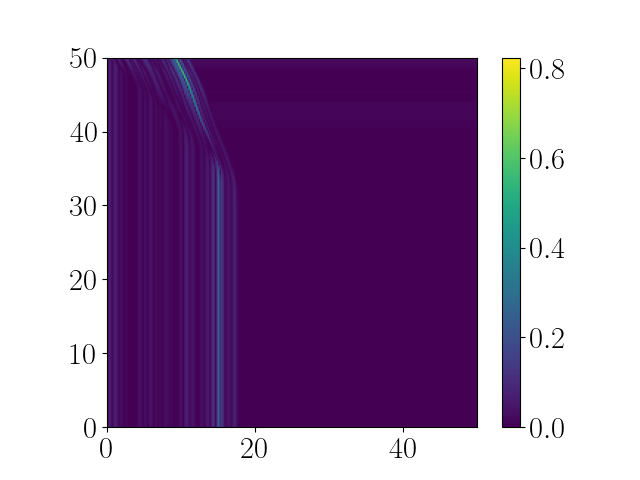}}
    \subfigure[Sparse]{\includegraphics[width=0.3\textwidth]{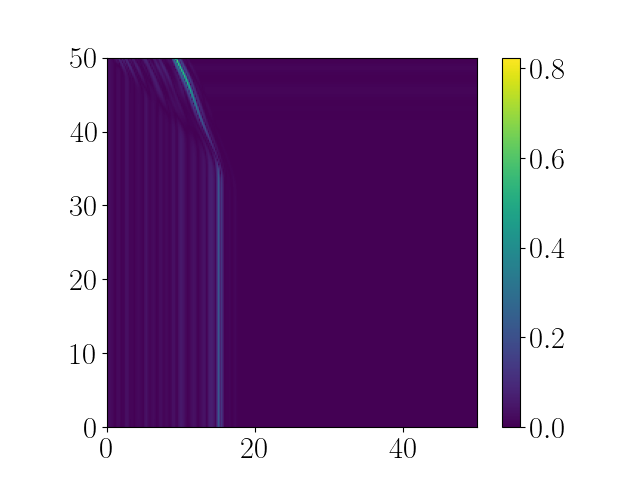}}
    \subfigure[CPN-S]{\includegraphics[width=0.3\textwidth]{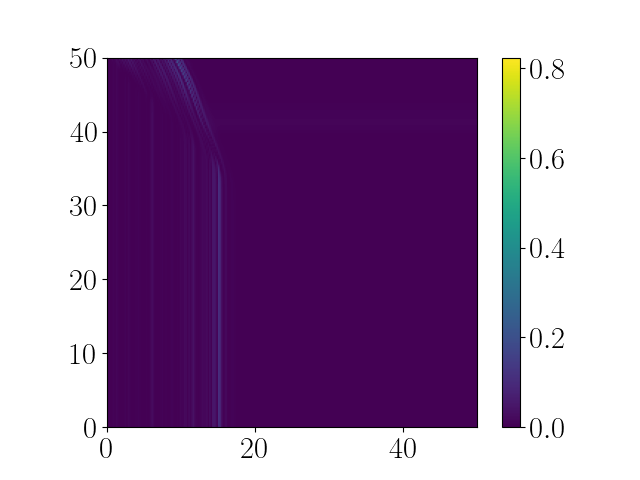}}
    \caption{(2D Burgers) Pointwise errors for $ n = 9 $,  $\lambda = 2.15717$, $t=9.7$.}
    \label{fig:burgers_wise_viz_solution}
\end{figure}


\section{Conclusion \red{and perspectives}}

\noindent

We have introduced a new method for the  construction  of an encoder-decoder pair for nonlinear manifold approximation. 
The encoder is a linear map associated with an orthogonal projection onto a low-dimensional space. 
The nonlinear decoder is a composition of structured polynomial maps  constructed using an adaptive strategy that ensures a manifold approximation with prescribed accuracy (in mean-squared or worst-case settings) and a control of the Lipschitz constant of the decoder. 
This yields an auto-encoder which is robust to perturbations in the data.   
The performance of the approach, compared to state of the art nonlinear model order reduction methods, has further been illustrated though numerical experiments. 
\red{
In the context of model order reduction for parameter-dependent models, where $M = \{u(y) : y \in Y\}$ is a set of solutions for parameters $y\in Y$, the above approach allows the offline construction of a low-dimensional manifold, which can then be used online in different ways, depending on the available information. 
An explicit approximation in terms of the parameters $y$ can be obtained by approximating offline   the coefficients $a(y) = E(u(y))$ as functions of $y$, from  the training set $((y_k , a(y_k)))_{k=1}^m$. A prediction is then obtained online without further information on the function to approximate. For the solution of a parameter-dependent equation for a new instance of parameter $y$, an approximation can also be obtained online by computing an approximation $D(a)$ which minimizes some residual norm, that requires access to the model. In the context of inverse problems, given some direct observations of a particular state $u(y)$ (with unknown value of parameters $y$), an approximation can  be obtained by computing an approximation $D(a)$ minimizing  some discrepancy between the prediction and the observations. The use of our manifold approximation method in the context of model order reduction is let for future research.}

\bibliographystyle{plain}
\bibliography{nonlinear-manifold-approximation}

\end{document}